\documentclass[a4paper,10pt]{article}
\usepackage{times, url}
\usepackage{amsfonts}

\usepackage{amsmath,amssymb,amsthm,amsfonts}
\usepackage{graphicx}
\usepackage{color}
\usepackage{mathrsfs}
\usepackage[utf8]{inputenc}
\usepackage[nottoc,numbib]{tocbibind}

\usepackage{amsfonts}

\usepackage{amsmath}
\usepackage{amstext}
\usepackage{amsopn}
\usepackage{amsbsy}
\usepackage{amscd}
\usepackage{amsxtra}
\usepackage{esint}
\usepackage{amsthm}

\usepackage{epsfig}
\usepackage{pstricks,multido,pst-node}

\usepackage[misc]{ifsym}

\usepackage{bibspacing}
\newtheorem{theorem}{Theorem}[section]
\newtheorem{lemma}[theorem]{Lemma}
\newtheorem{remark}{Remark}
\newtheorem{corollary}[theorem]{Corollary}
\newtheorem{proposition}[theorem]{Proposition}
\newtheorem{example}{ Example}
\newtheorem{definition}{Definition}

\theoremstyle{theorem}

\newtheorem*{notat}{Notations}

\newcommand{\ds}{\displaystyle}

\newcommand{\eps}{\epsilon}
\newcommand{\sech}{\mathrm{sech}}
\newcommand{\llang}{\left\lgroup}
\newcommand{\rrang}{\right\rgroup_{2}}

\newcommand{\langg}{\left\lgroup}
\newcommand{\rangg}{\right\rgroup_{2}}
\newcommand{\ranggg}{\right\rgroup_{1}}
\newcommand{\dxx}{\,\mathrm{d}{x}}
\newcommand{\dtt}{\,\mathrm{d}{t}}
\newcommand{\drr}{\,\mathrm{d}{r}}
\newcommand{\dsx}{\,\mathrm{d}\sigma_{x}}

\numberwithin{equation}{section}

    \newcommand\email[1]{\_email #1\q_nil}
    \def\_email#1@#2\q_nil{%
      \href{mailto:#1@#2}{{\emailfont #1\emailampersat #2}}
    }
    \newcommand\emailfont{\sffamily}
    \newcommand\emailampersat{{\color{cyan}\small@}}
		
		\usepackage{hyperref}
\hypersetup{
    colorlinks=true,
    linkcolor=blue,          
    citecolor=blue,        
    filecolor=magenta,      
    urlcolor=cyan           
}
 
\allowdisplaybreaks
\begin{document}
\setcounter{page}{1} 
\noindent{\LARGE\pmb{Domain-size effects on boundary layers of a nonlocal sinh-Gordon }\vspace{3pt}\\   \pmb{equation}}
\vspace{3mm}\\
{\large Chiun-Chang Lee\footnote{Institute for Computational and Modeling Science,  National Tsing Hua University,  Hsinchu 30013, Taiwan.\\
\Letter\hspace*{2.5mm}\email{chlee@mail.nd.nthu.edu.tw}}}





\vspace{2mm}

\noindent
\begin{abstract}
This work investigates a nonlocal sinh-Gordon equation with a singularly perturbed parameter in a ball. Under the Robin boundary condition, the solution asymptotically forms a quite steep boundary layer in a thin annular region, and rapidly becomes a flat curve outside this region. {Focusing more particularly on the structure of the thin annular layer in this region, the pointwise asymptotic expansion involving the domain-size is evaluated more sharply, where the domain-size exactly appears in the second term of the asymptotic expansion.} It should be stressed that the standard argument of matching asymptotic expansions is limited because the model has a nonlocal coefficient depending on the unknown~solution. A new approach relies on integrating ideas based on a Dirichlet-to-Neumann map in an asymptotic framework. The rigorous asymptotic expansions for the thin layer structure also matches well with the numerical results. Furthermore, various boundary concentration phenomena of the thin annular layer are precisely demonstrated.
\end{abstract}






\noindent

\section{The model and an overview}\label{sec-intro}
\noindent

Several important issues arising in plasma physics, electrochemistry and other topics lead to consider \textit{nonlocal} models with singularly perturbed parameters; see, e.g., \cite{hl2015-2,h2019-2,ks2015,l2014,l2016,l2019,lhl2016} and references therein. Focusing particularly on the electrochemical phenomena near the charged particle immersed in symmetrical electrolytes~\cite{hr2015,hy2020,s2012,w2014} as well as on related applications in colloidal systems~\cite{bm2018,hl2015-1,h2019-1,m2002,mhk2001}, we are interested in a nonlocal semi-linear equation
\begin{align}\label{eq1}
\eps^2\Delta{U}=\left(\fint_{\Omega}\cosh{U}\,\dxx\right)^{-1}{\sinh{U}}\quad\mathrm{in}\,\,\Omega,
\end{align}
and focus on a homogeneous Robin boundary condition 
\begin{align}\label{bd1}
U+\gamma\eps\partial_{\vec{n}}U=a\quad\mathrm{on}\,\,\partial\Omega.
\end{align}
Here $0<\eps\ll1$ is a singular perturbation parameter scaled by length
(see the related physical background below), $\Omega$ is a bounded smooth domain in $\mathbb{R}^N$ ($N>1$), $\Delta$ stands for the Laplace operator in $\mathbb{R}^N$,
 $\partial_{\vec{n}}:=\vec{n}\cdot\nabla$
and $\vec{n}:=\vec{n}(x)$ is the outward unit normal vector at $x\in\partial\Omega$ and
\begin{equation}\notag
 \fint_{\Omega}:=|\Omega|^{-1}\int_{\Omega}
\end{equation} 
with $|\Omega|$ the standard Lebesgue measure of $\Omega$. Besides, 
  $\gamma>0$ is a constant independent of $\eps$, and
$a:=a(x)\not\equiv0$ defined on $\partial\Omega$ is a smooth function independent of $\eps$.
It should be stressed that the nonlocal coefficient $\left(\fint_{\Omega}\cosh{U}\,\dxx\right)^{-1}$  
is a dimensionless variable because $\int_{\Omega}\cosh{U}\,\dxx$ 
has the same physical dimension as the volume.
Such a concept of dimensionless formulation plays a crucial role in 
connecting between the dimensionless model and the realistic physical phenomena; see, e.g., \cite{w2014}.

	\begin{figure}[htp]
\centering{%
\begin{tabular}{@{\hspace{-0pc}}c@{\hspace{-0pc}}c}
 \psfig{figure=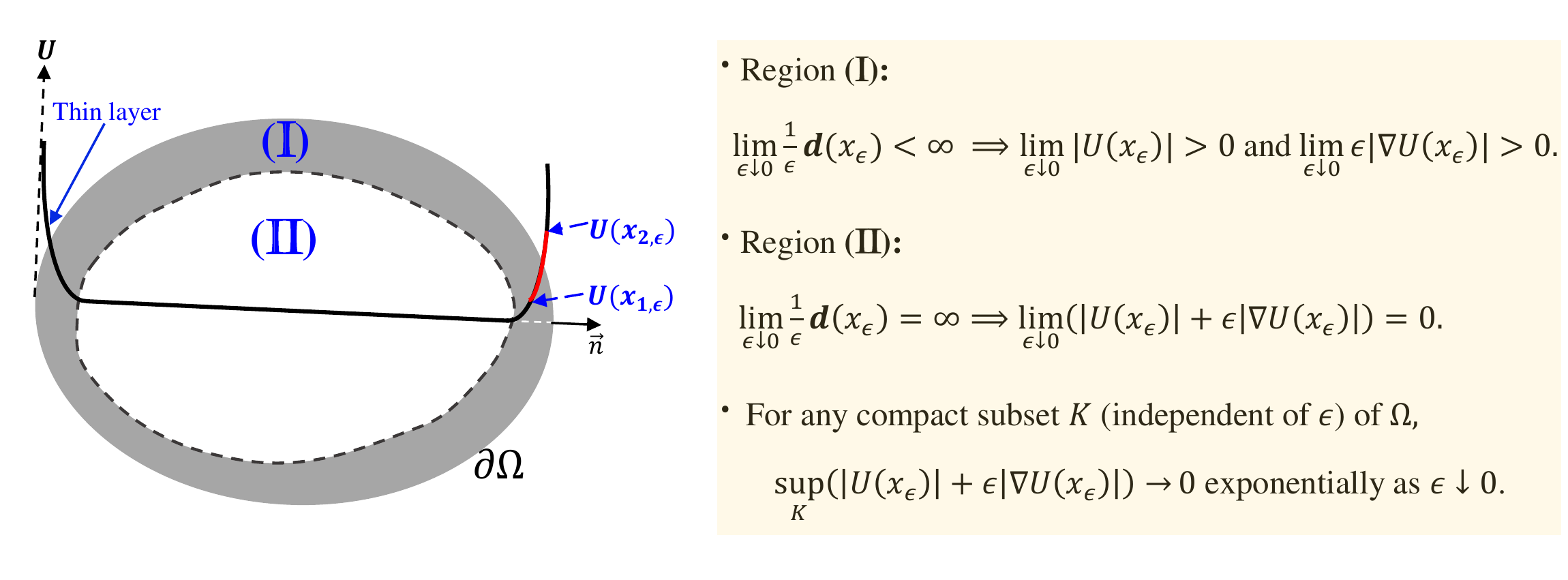, width=16cm}
  \end{tabular}
	}
	\caption{\small\em As $0<\eps\ll1$, $U$ develops a thin and quite steep layer near the boundary $\partial\Omega$;
	see Theorems~\ref{bdylayer-exist-thm} and \ref{new-cor}
in Section~\ref{sec-example} for the details of the layer structure.\label{fig-ccpb}}
\end{figure}

Equation~(\ref{eq1}) has various applications in the field of physics. 
When the nonlocal coefficient~$\left(\fint_{\Omega}\cosh{U}\,\dxx\right)^{-1}$ is withdrawn, (\ref{eq1}) becomes the standard
elliptic sinh--Gordon equation
 describing a system of interacting charged particles for the thermal equilibrium of plasma at very high temperature 
(corresponding to the parameter $\eps^{-2}$); see, e.g.,  \cite{jk1990}
and references therein. In such a situation, the physical background 
is usually set up in two dimensional domain~$\Omega\subset\mathbb{R}^2$.
Alternatively,
(\ref{eq1}) can be viewed as a 
sinh--Poisson equation endowed with 
 a ``minus sign" on its Laplace operator.
To distinguish between these models, in this case
we shall call (\ref{eq1}) a nonlocal sinh--Poisson type equation having a ``positive sign".
On the other hand, on a formal level of a ``stochastic" concept proposed in \cite{n2004}, 
(\ref{eq1}) can be rewritten as
independent identically distributed random variables with 
a Borel probability measure~$\mathcal{P}=\frac{1}{2}(\delta_{-1}+\delta_{+1})$ defined on~$[-1,1]$;
that is,
\begin{align*}
\eps^2\Delta{U}=2\left(\fint_{\Omega}\int_{[-1,1]}e^{{\mu}U}\mathcal{P}(\mathrm{d}\mu)\,\dxx\right)^{-1}\int_{[-1,1]}{\mu{e}^{\mu{U}}}
\mathcal{P}(\mathrm{d}\mu)\quad\mathrm{in}\,\,\Omega,
\end{align*}  
where $\delta_{-1}$ and $\delta_{+1}$
are Dirac delta functions concentrated at $-1$ and $+1$, respectively.
We further refer the reader to
 \cite{cd2001,GMN,MTW} and Section 3 of \cite{FLZ}
for related theories and applications of this model.

Besides its traditional applications, 
recently this model has been used to simulate the ion transport
and describe the structure and behavior of the thin electrical double layer (EDL) near the charged surface,
particularly for that of spherical colloidal particle in a symmetrical electrolyte solution (cf. \cite{s2012,w2014}); see Section~\ref{sec-1} for the specific detail. We will prove that the solution~$U$ with $0<\eps\ll1$ (corresponding to the electrostatic potential)
 is uniformly bounded to $\eps$ and exhibits boundary layers (corresponding to the EDL) with thickness
of the order $\eps$ near the boundary.

\textcolor{red}{A motivation of this study was raised by some boundary layer phenomena with the effects of the domain-size (cf. \cite{eg2019,kas2015,wm2011}) and the shape/geometry of the charged surface (cf. \cite{fq2011,m2002,mhk2001}), which have been observed numerically, but yet lack of understanding in the framework of rigorous mathematical analysis. It should be stressed that the effect of the domain-size on thin boundary layers might not be described clearly in general high-dimensional domains, so they particularly restricted themselves to one-dimensional models. As practical applications of electrochemistry, on the other hand, investigating the curvature effect on boundary layer structure usually focuses the physical domain on cylindrical, annular and spherical cases; see, e.g., \cite{fq2011} and references therein. Hence, based on these related investigations, one has a strong motivation to study (\ref{eq1})--(\ref{bd1}) with small~$\eps$ (corresponding to a small scaled Debye length), and $\Omega$ is set as a ball with the simplest geometry.}

\textcolor{red}{Note also that the sphere $\partial\Omega$ has constant mean curvature. Thus the effect of the boundary geometry on boundary layers in a spherical domain is just a special case, and the related result is not sufficiently understood when the spherical domain is replaced with a general bounded domain. Consequently, in such a situation we may focus our analysis on boundary layers with the domain-size effect, which is better to give expression to both the physical connection and the mathematical application; see Section~\ref{sec-setpro} for the setup. For the sake of convenience of the analysis and describing our results, in what follows we will use the radius of $\Omega$ to describe the effect of the domain-size (the diameter of $\Omega$) on the asymptotics of thin boundary~layers.}

Recently, there is a vast literature concerning standard elliptic sinh--Gordon type equations and sinh--Poisson type equations. However, for nonlocal model~(\ref{eq1})--(\ref{bd1})
with $0<\eps\ll1$,
to the best of our knowledge the related concentration phenomena and
the \textcolor{red}{domain-size effect} on the rigorous asymptotics of solutions remain unclear. \textcolor{red}{Since we focus mainly on $\Omega$ a ball,} our main interest will rely on the thin layer structure of solutions with the effect of \textcolor{red}{the radius of the spherical domain~$\Omega$}
and establish various boundary concentration phenomena. 
The main results are stated in Section~\ref{sec-example} and
their proofs are put in Sections~\ref{sec-curvature} and \ref{sec-concentration}.

	Before discussing the details of specific studies,
let us sketch the basic property of such thin layers
and point out the importance of analyzing its pointwise asymptotics~(cf. Figure~\ref{fig-ccpb}).
 Let $x_{1,\eps}$ and $x_{2,\eps}$ 
be two points located in this thin layer region
and lying on the same direction of the outward normal to the boundary. So we have 
$\ds\limsup_{\eps\downarrow0}{\eps}^{-1}\boldsymbol{{d}}{(x_{i,\eps})}<\infty$, $i=1,2$,
where $\boldsymbol{d}(\cdot):=\mathrm{dist}(\cdot,\partial\Omega)$ is the distance function to the boundary~$\partial\Omega$. Note also that each
$x_{i,\eps}$ approaches the boundary points
as $\eps$ goes to zero.
However, when $\ds\lim_{\eps\downarrow0}{\eps}^{-1}{\boldsymbol{d}(x_{1,\eps})}\neq\lim_{\eps\downarrow0}{\eps}^{-1}{\boldsymbol{d}(x_{2,\eps})}$, the height difference~$|U(x_{1,\eps})-U(x_{2,\eps})|$ of the thin layer profile at two points $x_{1,\eps}$ and $x_{2,\eps}$ does not tend to zero,
and the difference between the slopes of 
the thin layer profile at these two points (in the direction of the outward normal to the same boundary point) 
will tend to infinity.
Such a structure occurs in this quite thin region and is usually called the boundary layer.
Outside this thin region, the whole profile exponentially decay to zero as $\eps$ approaches zero.
Namely, the solution changes dramatically in this thin region, but merely makes a slight change outside this region.
Without the pointwise asymptotic analysis at $x_{\eps}$ where $\ds\limsup_{\eps\downarrow0}{\eps}^{-1}{\boldsymbol{d}(x_{\eps})}<\infty$, we merely obtain a ``one-point-jumping behavior" for the limiting profile of solutions at boundary points, and any information of the thin layer is hidden in the description.

Accordingly, we are devoted to pointwise asymptotics of solutions
in order to better understand the structure of the whole thin layer. 
We develop a singular perturbation analysis for radially symmetric solutions (in the case that the domain $\Omega$ is a ball
and $a(x)$ is a nonzero constant-valued function) and, more importantly, describe the effect of the \textcolor{red}{domain-size} on the thin layers precisely.
A series of basic estimates will be introduced in Sections~\ref{sec-dtn} and \ref{sec-curvature}.
The main concept is to establish
 a \textbf{Dirichlet-to-Neumann type map in an asymptotic framework} (cf. Theorem~\ref{DtoN-map}).
This rigorously derives the expansion formulas 
with \textbf{accurate first-two-term expansions} (with respect to $\eps$) for the layered solution
at each position which is sufficiently close to the boundary (in the sense that the distance between the point and the boundary
has at most the order $\eps$). Furthermore, we show
 in Proposition~\ref{prop-ceps} (see also, Lemma~\ref{dirac-lem}) that
the second order term (the small perturbation term) of the asymptotic expansions of the nonlocal coefficient plays a key role in the structure of the thin layer 
because it involves the \textcolor{red}{domain-size}. 
As will be clarified in Lemma~\ref{asy-radial-lem} and Theorem~\ref{new-cor},
the effect of the \textcolor{red}{domain-size} is significant in a thin region attaching to the boundary,
but is quite slight outside this thin region.

It should be stressed that the application of the
 Dirichlet-to-Neumann map to singularly perturbed nonlocal elliptic model is novel and different from the method of matching asymptotic expansions~(see, e.g., \cite{ar1981,f1973,h1979,o1968}).
 To the best of our knowledge, 
the traditional approach of matching asymptotic expansions is actually not easy to deal with such a nonlocal model 
before we obtain the accurate asymptoics of its nonlocal coefficients.
Accordingly, this new approach  has some advantages in dealing with such a singularly perturbed nonlocal model. 
We highlight them in turn here primarily for the reader to get a clear picture on this work.
\begin{itemize}
\item\,\,We first refer the reader to \cite{s2004},
showing that for some semilinear elliptic equations with Dirichlet boundary conditions in a bounded smooth domain,
 the mean curvature of domain boundary exactly appears in the second term of asymptotic expansions of their layers.
As a motivation, 
we consider the following nonlocal models
which are generalized from (\ref{eq1}) and study the structure of thin layers:
\begin{align}\label{0405-18eqn}
\eps^2\Delta{u}_i=\pmb{\pmb{\mathtt{C}}}_{u_i}^{\eps}f(u_i)\,\,\mathrm{in}\,\,\Omega
\end{align}
with the same boundary condition as (\ref{bd1}), $i=1,2$,
where $\pmb{\pmb{\mathtt{C}}}_{u_i}^{\eps}$ is a constant depending on unknown solution~$u_i$.
By \cite{s2004,t2001}, we assert that even if $\pmb{\pmb{\mathtt{C}}}_{u_1}^{\eps}-\pmb{\mathtt{C}}_{u_2}^{\eps}\to0$,
the different second order terms (tending to zero as $\eps\downarrow0$) of $\pmb{\mathtt{C}}_{u_1}^{\eps}$ and $\pmb{\mathtt{C}}_{u_2}^{\eps}$
results in different structures of their layers near the boundary.
However, as $\eps^2$ is sufficiently small, the numerical solutions are not easy to show the difference. 
Hence, for (\ref{eq1})--(\ref{bd1}) with small $\eps^2$,
 investigating the precise first two term of the nonlocal coefficient with respect to $\eps$
and establishing the pointwise asymptotics is usually of a challenge and particularly interesting.

\item\,\,In this work, we focus on the case of $\Omega=B_R:=\{x\in\mathbb{R}^N:|x|<R\}$ a ball
with the simplest geometry
and $a(x)\equiv{a}_0\neq0$ a constant-valued function (cf. Section~\ref{sec-setpro}). 
Then the uniqueness of (\ref{eq1})--(\ref{bd1}) (see Proposition~\ref{corapp} in the Appendix)
implies that $U$ is radially symmetric in $B_R$.
We develop a rigorous asymptotic analysis based on a Dirichlet-to-Neumann map 
in the asymptotic framework with $0<\eps\ll1$ (cf. (\ref{amen}) and Theorem~\ref{DtoN-map}).
Using such an approach, we establish precise first two terms of the nonlocal coefficient of (\ref{eq1}). 
In particular, the second term exactly involves \textcolor{red}{the curvature  of $\partial\Omega$ given by $R^{-1}$} (cf. Proposition~\ref{prop-ceps}). 
Furthermore, we derive an ODE of $U$ in an asymptotic framework
 involving the \textcolor{red}{diameter $2R$ of $\Omega=B_R$} (cf. Lemma~\ref{lem-1119-9011}).
We show that as $\eps\downarrow0$, $U$ develops quite steep boundary layers in a thin region with thickness of the order $\eps$ attaching to the boundary~$\partial{\Omega}$ (cf. Figure~\ref{fig-ccpb} and Theorem~\ref{bdylayer-exist-thm}). 
We completely study the structure of the thin layer through establishing refined pointwise asymptotics of $U$ in this thin region
(cf. Theorem~\ref{new-cor} and the proof in Section~\ref{sec-mainthm-2}). An interesting outcome shows that
 the second order term of the asymptotics of $U(x_{\eps})$
 is algebraically dependent on the first two order terms (with respect to $\eps$) of
$\boldsymbol{{d}}{(x_{\eps})}$, which are presented in (\ref{0308-afr}) and Theorem~\ref{new-cor}.

\item\,\,Under the same boundary condition, a comparison between asymptotic solutions of the nonlocal sinh-Gordon equation
and the standard sinh-Gordon equation is completely studied.
Although these two solutions have the same leading order terms, their second order terms are totally different. The main difference comes from the second order term of the asymptotic expansion of the nonlocal coefficient of (\ref{eq1}) (see Section~\ref{sec-comparison}). The conclusion supports the above assertion. We also want to point out that the numerical solutions of these two models with $\eps=10^{-3}$ seem almost overlapping near the boundary (see Figure~\ref{chlee04072018} in Section~\ref{sec-comparison}). However, a closer look at pointwise asymptotics of solutions reveals that the slopes of their solution curves near the boundary always have $\mathcal{O}(1)$ difference
which does not tend to zero as $\eps$ goes to zero (see Remark~\ref{0403-rk-2018}).

\item\,\,Various boundary concentration phenomena for the thin boundary layer are established (cf. Theorem~\ref{mainthm-1} and the proof in Section~\ref{sec-concentration}).
\end{itemize}

We shall emphasize that although this work focuses mainly on nonlocal sinh-Gordon equations of radial cases, the analysis technique can be generalized to a class of nonlocal elliptic equations~(\ref{0405-18eqn})
with $\pmb{\mathtt{C}}_{u_i}^{\eps}=\left(\int_{\Omega}F(u_i)\dxx\right)^{l}$ for positive function $F$ and $l\neq0$, 
which is one of our ongoing projects.

\textcolor{red}{Although the progress has been made in this paper along with the radial case, there were still some interesting questions left open and we need to explain the difficulty for further pursues. Precisely speaking, for the general bounded smooth domain $\Omega$, it seems to be a challenge problem about the effect of the boundary geometry on thin boundary layers of equation \eqref{eq1} with homogeneous mixed boundary conditions like~\eqref{bd1}, for which there are two difficulties so that we still do not have any satisfactory result. One difficulty lies in the fact that under the homogeneous mixed boundary condition, the solution is in general not a constant on the boundary unless it is an overdetermined problem. (In this situation, the domain $\Omega$ is unknown to be determined.) Hence, the standard blow-up argument for boundary-layer problems with the homogeneous Dirichlet boundary condition~(see, e.g., \cite{s2004}) cannot be rigorously applied to this case. On the other hand, the nonlocal coefficient like that in \eqref{eq1} depends on the unknown $u$ and it is expected that the refined asymptotics of the nonlocality with respect to $0<\eps\ll1$ may affect the second-order term of the asymptotic expansion of the thin boundary layer, which can be observed in the radial case; see, e.g., Proposition~\ref{prop-ceps}. However, to the best of our knowledge, such a nonlocal effect on thin layers in general domains seems to not be found in the related literature. For the problem mentioned the above, we will keep working on it in a forthcoming project.}

\textcolor{red}{The rest of the work goes as follows. Based on the background and the motivation mentioned in this section, in the next section we formally formulate the problem with notations and definitions and collect Proposition~\ref{prop-ceps},  Theorems~\ref{bdylayer-exist-thm}--\ref{new-cor} (about domain-size effects on the thin layers with pointwise asymptotics) and Theorem~\ref{mainthm-1} (about the boundary concentration phenomenon of the thin layer) as the  main results. In Section~\ref{sec-dtn} we establish a Dirichlet-to-Neumann map in an asymptotic framework at boundary points, and complete the proof of  Proposition~\ref{prop-ceps}. These materials are crucial for investigating the refined asymptotics of thin boundary layer of the nonlocal model.  In~Section~\ref{sec-curvature} we state the proof of Theorems~\ref{bdylayer-exist-thm} and \ref{new-cor}. In particular, in~Section~\ref{sec-comparison} we compare the difference between asymptotics of solutions to nonlocal and standard elliptic sinh--Gordon equations. The proof of Theorem~\ref{mainthm-1} is stated in Section ~\ref{sec-concentration}. Finally, in the Appendix we prove the uniqueness of the solution of \eqref{eq1} with three type boundary~conditions.}



\section{Problem formulation and the results}\label{sec-1}
\noindent

 Let us start with an energy functional  
\begin{align}\label{energy0}
E_{\eps}[U]=\frac{\eps^2}{2}\int_{\Omega}|\nabla{U}|^2\dxx+|\Omega|\log\fint_{\Omega}\cosh{U}\,\dxx+\frac{\eps}{2\gamma}\int_{\partial\Omega}(U-a)^2\dsx,\,\,U\in{H}^1(\Omega).
\end{align}
The singular perturbation parameter~$\eps$ can be regarded as a length-scale parameter.
Thus, the standard dimension analysis immediately implies that the boundary term
 $\frac{\eps}{2\gamma}\int_{\partial\Omega}(U-a)^2\dsx$
scales in the same way as the gradient term~$\frac{\eps^2}{2}\int_{\Omega}|\nabla{U}|^2\dxx$
and the logarithm term~$|\Omega|\log\fint_{\Omega}\cosh{U}\,\dxx$.

(\ref{eq1})--(\ref{bd1}) results from applying variational calculus to functional $E_{\eps}$ over ${H}^1(\Omega)$,
where the nonlocal form
is obtained from the variation of the logarithm term of (\ref{energy0}).
Indeed, functional $E_{\eps}$
is strictly convex and admits a unique minimizer in $H^1(\Omega)$ (cf. Proposition~\ref{convexthm} in the Appendix).
Performing the variation of (\ref{energy0}) and applying the direct method yields that
 the unique minimizer~$U$ is a weak solution 
of (\ref{eq1}) with the Robin boundary condition~(\ref{bd1}). 
Furthermore,
note that $\Omega$ is a bounded smooth domain. Applying the standard elliptic regularity theory and the Sobolev's embedding argument~(see, e.g., \cite{GT2001})
concludes that the unique minimizer of the energy functional~(\ref{energy0}) is a classical solution of (\ref{eq1})--(\ref{bd1}).
For the sake of completeness, we prove the uniqueness of the classical solution to (\ref{eq1}) with the boundary condition~(\ref{bd1}) in Proposition~\ref{corapp}(i). On the other hand, we also prove the uniqueness for the classical solutions of
 the equation~(\ref{eq1}) with the Dirichlet and the Neumann boundary conditions which are stated in  
Proposition~\ref{corapp}(ii) and~(iii).

Equation~(\ref{eq1})--(\ref{bd1}) has important applications in electrochemistry, biology and physiology.
In the ion-conserving Poisson--Boltzmann theory for symmetrical electrolytes~\cite{s2012},
equation~(\ref{eq1}) has been derived under the assumption that \textit{the total density of all ion species are conserved}.
Here $U$ corresponds to the electrostatic potential, and the parameter $\eps$ is 
a scaled Debye screening length~\cite{lhl2011,lhl2016}. Physically, $\Omega$ usually represents the bulk in which all ion species occupy, where
$\frac{1}{2}\left(\fint_{\Omega}\cosh{U}\,\dxx\right)^{-1}e^{U}$ corresponds to the Boltzmann distribution of anion species with charge valence $-e_0$ ($e_0$ is the elementary charge), and $\frac{1}{2}\left(\fint_{\Omega}\cosh{U}\,\dxx\right)^{-1}e^{-U}$ corresponds to the Boltzmann distribution of cation species with charge valence $+e_0$.
The boundary~$\partial\Omega$ is regarded as a charged surface. 
Moreover, the electric field driving the ions toward the charged surface creates the EDL.
The Robin boundary condition~(\ref{bd1}) is derived from the capacitance effect of the EDL~\cite{k2007-1}, where
 $\gamma\eps$ is a scaled length with respect to the Stern layer, and $a:=a(x)$ is an extra potential applied on the charged surface~$\partial\Omega$.
In recent years, this model is used to simulate the behavior of the electrostatic potential in the EDL, and has many applications in colloidal systems. Hence, a boundary layer problem for the model~(\ref{eq1})--(\ref{bd1}) naturally arises in mathematics, 
and the rigorous analysis seems a challenge.
According to this motivation, we are interested in the boundary layer problem for model~(\ref{eq1})--(\ref{bd1}),
especially in the boundary concentration phenomena and the pointwise description of the thin layer structure.

It is worth stressing a similar model proposed in~\cite{hy2020,lhl2016,lr2018,rlw2006,w2014}, e.g.,
\begin{align}\label{eqccpb}
\eps^2\Delta{U}=\left(\fint_{\Omega}e^{U}\,\dxx\right)^{-1}{e^{U}}-\left(\fint_{\Omega}e^{-U}\,\dxx\right)^{-1}{e^{-U}}\quad\mathrm{in}\,\,\Omega.
\end{align}
This model is a steady-state Poisson--Nernst--Planck equation
for symmetric $1:1$ electrolytes, 
assuming that \textit{the density of each ion species is conserved} (cf. \cite{rl2007,w2014}).
Accordingly, the physical setting of model~(\ref{eq1}) is different from that of~(\ref{eqccpb}). On the other hand,
from a mathematical perspective one finds that (\ref{eq1}) does not satisfy the shift invariance and
 the integral of its right-hand side $\left(\fint_{\Omega}\cosh{U}\,\dxx\right)^{-1}{\sinh{U}}$ over $\Omega$ is not a constant value. Such a property is totally different from that of~(\ref{eqccpb}),
and may increase the difficulty on the analysis of solutions.
In the present work, new analysis technique is developed to
deal with the asymptotic behavior of solutions of (\ref{eq1})--(\ref{bd1}) with small $\eps>0$.

\subsection{The radial configuration and preliminary techniques}\label{sec-setpro}
\noindent

 For equation~(\ref{eq1})--(\ref{bd1}), the asymptotics of the nonlocal coefficient $\left(\fint_{\Omega}\cosh{U}\,\dxx\right)^{-1}$
may depend on the domain geometry. To see such effects in a simple way,
we focus mainly on the case that $\Omega$ is a ball with the simplest geometry, 
and establish fine asymptotic expansions with the effect of \textcolor{red}{the domain-size} on the thin layer as $\eps$ approaches zero. 
This setup describes a realistic electrolyte involving, for example, 
electrostatic interactions in spherical colloidal systems;
see the physical background in, e.g., \cite{m2002,mhk2001,w2014} and references therein.
Mathematically, such a setup allows us to study radially symmetric solutions where precise estimates are
more readily available. 

Hence, we may set $\Omega=B_R:=\{x\in\mathbb{R}^N:|x|<R\}$ for $R>0$, and $a(x)\equiv{a}_0$
on $\partial{B}_R$, where $a_0\in\mathbb{R}$ is a constant. In what follows we let the surface area of the unit sphere $|\partial{B}_1|=1$ for the convenience. Then the uniqueness for solutions of (\ref{eq1})--(\ref{bd1}) (cf. Proposition~\ref{corapp}) asserts that
 $U(x)=u(r)$ with $r=|x|$ is radially symmetric and
(\ref{eq1})--(\ref{bd1}) is equivalent ~to 
\begin{align}
\eps^2\left(u''(r)+\frac{N-1}{r}u'(r)\right)\,=\,\pmb{\mathtt{C}}(u)\sinh{u},&\quad{r}\in(0,R),\label{eq2}\\
\pmb{\mathtt{C}}(u)\,=\,\left(\frac{N}{R^N}\int_0^Rs^{N-1}\cosh{u}(s)\,\mathrm{d}s\right)^{-1},&\label{cons2}\\
u'(0)=0,\quad{u}(R)+{\gamma\eps}u'(R)={a}_0.&\label{bd2}
\end{align}
The solution $u$ may depend on the parameter $\eps$ and should be denoted as $u_{\eps}$ but we denote it as $u$ for
a sake of simplicity. Note also that coefficient $\pmb{\mathtt{C}}(u)$ depending on $u$ is unkown.

When $a_0=0$,  (\ref{eq2})--(\ref{bd2}) merely has a trivial solution due to the uniqueness. 
To avoid the trivial case, without loss of generality we may assume $a_0>0$. 
We are devoted to the pointwise asymptotics and various boundary concentration phenomena
of the solution $u$ as $0<\eps\ll1$.

In order to properly state the main results, we now introduce some notational conventions and definitions that will be used throughout the whole paper.
\begin{notat}~\\
$\bullet$~We abbreviate 
``~$\leq{C}$ "~ to `` $\lesssim$ ", where $C>0$ is a generic constant independent of parameter~$\eps$.\\
$\bullet$~$\mathcal{O}(1)$ is denoted by a bounded quantity independent of $\eps$.\\
$\bullet$~$o_{\eps}(1)$  is denoted by a small quantity tending towards zero as $\eps$ approaches zero.
\end{notat}

We can now make the following definitions.
\begin{definition}
Assume that $f_{\eps}$ has an expansion $f_{\eps}=\displaystyle\sum_{i\in\mathbb{N}}f_{(i)}\eps^{\sigma_i}$, where $f_{(i)}$ and $\sigma_i$
are real numbers independent of $\eps$ and $\sigma_i<\sigma_{i+1}$.
We define 
\begin{align}\label{mapping-t}
\langg{f}_{\eps}\ranggg:=f_{(1)}\eps^{\sigma_1}\quad\mathrm{and}\quad
\langg{f}_{\eps}\rangg:=f_{(1)}\eps^{\sigma_1}+f_{(2)}\eps^{\sigma_2} 
\end{align}
 which map $f_{\eps}$ to its leading term and first two terms, respectively.
\end{definition}

Next, to demonstrate the boundary concentration phenomena, we introduce a Dirac delta function~$\delta_R$ concentrated at the boundary point $r=R$ as follows.
\begin{definition}\label{def1}
It is said that 
\begin{equation*}
 f_\epsilon\rightharpoonup\mathcal{C}\delta_R\,\,\mathrm{weakly\,\,in}\,\,
\mathrm{C}([0,R];\mathbb{R})
\end{equation*}
with a weight $\mathcal{C}\neq0$ as $\epsilon\downarrow0$ if there holds
\begin{align}\notag
\lim_{\epsilon\downarrow0}\int_0^Rh(r)f_\epsilon(r)\drr=\mathcal{C}h(R)
\end{align}
for any continuous function $h:[0,R]\to\mathbb{R}$ independent of $\epsilon$.
\end{definition}


Since $\pmb{\mathtt{C}}(u)$ is positive and $\sinh{u}$ is strictly increasing to $u$,
applying the standard elliptic PDE comparison to (\ref{eq2})--(\ref{bd2}),
we obtain that $u$ and $u'$ exponentially decay to zero in the interior domain $(0,R)$ as $\eps\downarrow0$.
One key point for studying boundary asymptotics of $u$ is to transform (\ref{eq2}) into an integro-differential equation
\begin{align}\label{cr-in3}
\frac{\eps^2}{2}u'^2(t)+(N-1)\eps^2\int_{\frac{R}{2}}^t\frac{1}{r}u'^2(r)\drr=\pmb{\mathtt{C}}(u)\cosh{u}(t)+\mathtt{K}_{\eps},\,\,t\in[0,R),
\end{align}
where $\mathtt{K}_{\eps}$ is a constant depending on $\eps$. Obviously, 
using the boundary condition~(\ref{bd2}) and (\ref{cr-in3}), 
we can make appropriate manipulations to obtain $\pmb{\mathtt{C}}(u)\to1$ and $\mathtt{K}_{\eps}\to-1$  (as $\eps\downarrow0$)
and the exact leading-order terms of 
 boundary asymptotic expansions of $u(R)$ and $u'(R)$ (see, e.g., the argument in \cite{lhl2011,lhl2016}). However, the leading order terms cannot show
the effect of \textcolor{red}{the domain-size $2R$ (and also the boundary curvature $R^{-1}$)} on the 
solution structure.
To basically understand such an issue,
investigating their first two term asymptotic expansions with respect to $\eps$
is necessary.

There are two main difficulties requiring discussion. 
The first difficulty comes from a fact that $\pmb{\mathtt{C}}(u)$ depends on the unknown solution $u$.
Hence, as $\eps$ approaches zero, the asymptotics of $u$ and $\pmb{\mathtt{C}}(u)$
are influenced by each other. Such rigorous analysis will be clarified in Section~\ref{sec-dtn}.
Particularly, for (\ref{cr-in3}),
we show in Lemma~\ref{cruc-in-thm} that
 the leading order term of $(N-1)\eps^2\int^R_{R/2}\frac{1}{r}u'^2(r)\drr$ exactly determines the second order term (with respect to $\eps$) of $\pmb{\mathtt{C}}(u)$, $u(R)$ and $u'(R)$ as $0<\eps\ll1$.
Based on such an observation,
it suffices to establish the exact leading order term of $\int_{0}^Rg(r)\cdot\eps{u}'^2(r)\drr$ for any continuous function $g\in\mathrm{C}([0,R])$. An interesting outcome shows that $\eps{u}'^2$ 
behaves exactly as a \textbf{Dirac delta function} concentrated at boundary point $r=R$  (cf. Lemma \ref{dirac-lem}).

The other difficulty comes from the Robin boundary condition~(\ref{bd2}) at $r=R$.
As a technical idea for dealing with the asymptotics of the thin layer near the boundary $r=R$, we establish
a \textbf{Dirichlet-to-Neumann type map} (cf. Theorem~\ref{DtoN-map}), 
\begin{equation*}
\Lambda_{\eps}: u(R) \mapsto{u}'(R)
\end{equation*}
which maps $u(R)$ to $u'(R)$ in an asymptotic framework,
\begin{align}\label{amen}
\langg\Lambda_{\eps}(u(R))\rangg=&\frac{2}{\eps}\langg\sinh\frac{u(R)}{2}\rangg
-\underbrace{\frac{2}{R}\left(N\cosh^2\frac{\langg{u}(R)\ranggg}{2}-1\right)\tanh\frac{\langg{u}(R)\ranggg}{4}}_{\mathcal{O}(1)\,\,\mathrm{term\,\,involving\,\,the\,\,\textcolor{red}{domain\,\,size}\,\,effect}},
\end{align}
as $0<\eps\ll1$.
Moreover,
\begin{equation*}
\left|u'(R)-\langg\Lambda_{\eps}(u(R))\rangg\right|\lesssim\sqrt{\eps}.
\end{equation*}
We stress that $\langg\sinh\frac{u(R)}{2}\rangg$ involves the second order term of $u(R)$.
Combining (\ref{amen}) with the Robin boundary condition~(\ref{bd2}),
we can determine the exact first two order expansions of $\pmb{\mathtt{C}}(u)$, $u(R)$ and $u'(R)$
with respect to $\eps$, which are described as follows.

\begin{proposition}\label{prop-ceps}
For $\eps>0$, let $u$ be the unique classical solution of (\ref{eq2})--(\ref{bd2}), where $a_0$ and $\gamma$ are positive constants independent of $\eps$. Then 
as $0<\eps\ll1$, we have
\begin{align}
\langg\pmb{\mathtt{C}}(u)\rangg=&\,1-\frac{2N}{R}\left(\cosh\frac{b}{2}-1\right)\eps,\label{ceps-1}\\
\langg{u}(R)\rangg=&\,b+\frac{2}{R}\eps\cdot\frac{\gamma\left(N\cosh^2\frac{b}{2}-1\right)\tanh\frac{b}{4}}{\gamma\cosh\frac{b}{2}+1},\label{mainth1-id6}\\
\langg{u}'(R)\rangg=&\,\frac{2}{\eps}\sinh\frac{b}{2}-\frac{2}{R}\cdot\frac{\left(N\cosh^2\frac{b}{2}-1\right)\tanh\frac{b}{4}}{\gamma\cosh\frac{b}{2}+1}\label{mainth1-id5}
\end{align}
with an optimal error estimate
\begin{equation}\label{11-0305}
\left|\pmb{\mathtt{C}}(u)-\langg\pmb{\mathtt{C}}(u)\rangg\right|+\left|{u}(R)-\langg{u}(R)\rangg\right|+\eps\left|{u}'(R)-\langg{u}'(R)\rangg\right|\lesssim\eps^{3/2},
\end{equation}
where $b\in(0,a_0)$ uniquely solves
\begin{align}\label{arbequ}
b+2\gamma\sinh\frac{b}{2}=a_0.
\end{align}
\end{proposition}

Note that Proposition~\ref{prop-ceps} precisely illustrates the effects of the coefficient $\gamma$ and the \textcolor{red}{domain-size} on the boundary asymptotic expansions of $u$. Particularly, the \textcolor{red}{domain-size} exactly appears in their second order terms \textcolor{red}{with the precise coefficients}, and the third order terms of $\pmb{\mathtt{C}}(u)$, $\eps^{-1}u(R)$ and $u'(R)$ tend to zero as $\eps\downarrow0$, \textcolor{red}{which sufficiently implies how the domain-size influences the layer structure including the slope near the boundary. As an example, we fix $R_{\star}>0$ and let $u_i$ correspond to the unique classical solution of (\ref{eq2})--(\ref{bd2}) with $R=R_i$, $i=1,2$, where ${R}_{\star}<R_1<R_2$. Then for $b>0$ and $\gamma>0$, there exists $\eps_{\star}>0$ depending mainly on $R_{\star}$ such that as $0<\eps<\eps_{\star}$, we have $\pmb{\mathtt{C}}(u_1)<\pmb{\mathtt{C}}(u_2)$, $u_1(R_1)>u_2(R_2)$ and the slopes of boundary layers of $u_1$ and $u_2$ at their boundary points have a bit difference affected by the domain-size:
\begin{align*}
u_2'(R_2)-u_1'(R_1)\approx{C_{\star}}\left(\frac{1}{R_1}-\frac{1}{R_2}\right)\quad\mathrm{as}\quad0<\eps\ll1,
\end{align*}
where $C_{\star}=\frac{2\left(N\cosh^2\frac{b}{2}-1\right)\tanh\frac{b}{4}}{\gamma\cosh\frac{b}{2}+1}$.
We should stress again the importance of \eqref{ceps-1}--\eqref{mainth1-id5} since their leading terms merely reveal the effect of $\gamma$, but do not have the domain-size effect on solutions. Moreover, we will show in Corollary~\ref{cor-0403uv} for the difference of asymptotics between the nonlocal sinh-Gordon equation~\eqref{eq1}--\eqref{bd1} and the standard sinh-Gordon equation~\eqref{v-eqn}--\eqref{v-bdy} due to the domain-size effect.}

Note also that Proposition~\ref{prop-ceps} indicates the existence of boundary layer. To understand the refined structure of the boundary layer, we shall further consider a quite thin region attaching to the boundary:
\begin{align}\label{omepareps}
\pmb{\mathbb{B}_{\partial}^{\eps}}:=\bigg\{r_{\eps}\in[0,R]:&\,\frac{R-r_{\eps}}{\eps}=\mathtt{p}+o_{\eps}(1)\quad\mathrm{for\,\,some}\,\,\mathtt{p}\geq0\,\,\mathrm{independent\,\,of}\,\,\eps\bigg\}.
\end{align}
We are interested in the pointwise asymptotics of the boundary layer with the effect of the domain-size in $\pmb{\mathbb{B}_{\partial}^{\eps}}$. The main results are introduced in Section~\ref{sec-example}.

\subsection{Statement of the main theorems}\label{sec-example}
\noindent

The following theorem makes a specific presentation to assert that
 $u$ indeed exhibits a quite steep boundary layer in the whole region of $\pmb{\mathbb{B}_{\partial}^{\eps}}$ as $\eps\downarrow0$,
which is in extreme contrast with the behavior of $u$ in the region $[0,R]-\pmb{\mathbb{B}_{\partial}^{\eps}}$.
\begin{theorem}\label{bdylayer-exist-thm}
For $\eps>0$, let $u$ be the unique classical solution of (\ref{eq2})--(\ref{bd2}), where $a_0$ and $\gamma$ are positive constants independent of $\eps$. Then for $r_{\eps}\in[0,R]$, 
\begin{align}\label{0315-2018}
\limsup_{\eps\downarrow0}\frac{R-r_{\eps}}{\eps}<\infty\,\,i{\!}f\,\,and\,\,only\,\,i{\!}f\,\,
\begin{cases}
\ds\liminf_{\eps\downarrow0}u(r_{\eps})>0,\\
\ds\liminf_{\eps\downarrow0}{\eps}u'(r_{\eps})>0.
\end{cases}
\end{align}
\end{theorem}
The proof of Theorem~\ref{bdylayer-exist-thm} is stated in Section~\ref{sec-pf-0315-2018}.

Moreover,
to get the refined structure of the thin layer in $\pmb{\mathbb{B}_{\partial}^{\eps}}$, 
we focus on those points $r^{\eps}_{\mathtt{p};\mathtt{q}}\in\pmb{\mathbb{B}_{\partial}^{\eps}}$ satisfying
\begin{align}\label{0308-afr}
\llang\frac{R-r^{\eps}_{\mathtt{p};\mathtt{q}}}{\eps}\rrang=\mathtt{p}+\frac{\mathtt{q}}{R}\eps,\,\,\mathrm{where}\,\,\mathtt{p}\geq0\,\,\mathrm{and}\,\,\mathtt{q}\in\mathbb{R}\,\,\mathrm{are\,\,independent\,\,of}\,\,\eps.
\end{align}
The setting of (\ref{0308-afr}) with specific orders of $\eps$ is mainly due to
the boundary asymptotic expansions of $u(R)$ and $u'(R)$ in Proposition~\ref{prop-ceps}
so that we can compare them with $u(r^{\eps}_{\mathtt{p};\mathtt{q}})$
and $u'(r^{\eps}_{\mathtt{p};\mathtt{q}})$ in a direct way.
The following theorem reveals that the leading order terms of $u(r^{\eps}_{\mathtt{p};\mathtt{q}})$
and $u'(r^{\eps}_{\mathtt{p};\mathtt{q}})$
are uniquely determined by $\mathtt{p}$ and the second order terms of that
depend on both $\mathtt{p}$ and $\mathtt{q}$. Moreover,  the effect of the \textcolor{red}{domain-size}
appearing in their second order terms are precisely described.

\begin{theorem}[Pointwise descriptions with \textcolor{red}{domain-size}  effects in $\pmb{\mathbb{B}_{\partial}^{\eps}}$]\label{new-cor}
Under the same hypotheses as in Theorem~\ref{bdylayer-exist-thm}, as $0<\eps\ll1$,
for $r^{\eps}_{\mathtt{p};\mathtt{q}}\in\pmb{\mathbb{B}_{\partial}^{\eps}}$ obeying (\ref{0308-afr}), 
the precise first two terms of $u(r^{\eps}_{\mathtt{p};\mathtt{q}})$ and $u'(r^{\eps}_{\mathtt{p};\mathtt{q}})$
are depicted as follows: 
\begin{align}
&\langg{u}(r^{\eps}_{\mathtt{p};\mathtt{q}})\rangg
=\,k(\mathtt{p})+\frac{\eps}{R}{\mathcal{H}^{\gamma;b}_{\mathtt{p};\mathtt{q}}}\sinh\frac{k(\mathtt{p})}{2},\label{u-0d}\\
\langg{u}'(r^{\eps}_{\mathtt{p};\mathtt{q}})\rangg=&\,2\sinh\frac{k(\mathtt{p})}{2}\cdot\left[\frac{1}{\eps}-\frac{1}{R}\left(2N\sinh^2\frac{b}{4}+\frac{N-1}{2}\sech^2\frac{k(\mathtt{p})}{4}-\frac{\mathcal{H}^{\gamma;b}_{\mathtt{p};\mathtt{q}}}{2}\cosh\frac{k(\mathtt{p})}{2}\right)\right],\label{u-1std}
\end{align}
 where $k(\mathtt{p})\in(0,b]$ is uniquely determined by
\begin{align}\label{0308-k}
\left(1+\frac{N-1}{2R}\right)\log\frac{\tanh\frac{b}{4}}{\tanh\frac{k(\mathtt{p})}{4}}+\frac{N-1}{4R}\left(\tanh^2\frac{k(\mathtt{p})}{4}-\tanh^2\frac{b}{4}\right)=\,\mathtt{p},
\end{align}
and
\begin{align}\label{mathcal-H}
\mathcal{H}^{\gamma;b}_{\mathtt{p};\mathtt{q}}:=\frac{\gamma\left(N\cosh^2\frac{b}{2}-1\right)\sech^2\frac{b}{4}}{\gamma\cosh\frac{b}{2}+1}\cdot\frac{1+\frac{N-1}{2R}\sech^2\frac{b}{4}}{1+\frac{N-1}{2R}\sech^2\frac{k(\mathtt{p})}{4}}-\frac{2\mathtt{q}-4N\mathtt{p}\sinh^2\frac{b}{4}}{1+\frac{N-1}{2R}\sech^2\frac{k(\mathtt{p})}{4}}.
\end{align}
Moreover, 
the convergence
\begin{align}\label{0403-2018}
\frac{1}{\eps}\left|u(r^{\eps}_{\mathtt{p};\mathtt{q}})-\langg{u}(r^{\eps}_{\mathtt{p};\mathtt{q}})\rangg
\right|+\left|u'(r^{\eps}_{\mathtt{p};\mathtt{q}})-\langg{u}'(r^{\eps}_{\mathtt{p};\mathtt{q}})\rangg\right|\stackrel{\eps\downarrow0}{-\!\!\!-\!\!\!\rightarrow}0
\end{align}
is uniformly as $\mathtt{p}$ is located in a bounded subinterval of $[0,\infty)$.  
\end{theorem}

The uniqueness of (\ref{0308-k}) is trivially due to the fact that 
$\left(1+\frac{N-1}{2R}\right)\log\frac{\tanh\frac{b}{4}}{\tanh\frac{k}{4}}+\frac{N-1}{4R}\left(\tanh^2\frac{k}{4}-\tanh^2\frac{b}{4}\right)$ is strictly decreasing to $k$ in $(0,b]$.
The proof of Theorem~\ref{new-cor} is stated in Section~\ref{sec-mainthm-2}.

Theorem~\ref{new-cor} establishes a rigorous analysis technique for rendering the \textcolor{red}{domain-size} effect
on the thin boundary layer of $u$. Moreover, we can calculate
the precise first two terms of $\langg\frac{\mathrm{d}^nu}{\mathrm{d}r^n}(r^{\eps}_{\mathtt{p};\mathtt{q}})\rangg$ ($n\geq2$),
where the leading term is the order of $\eps^{-n}$
and the second order term $\mathcal{O}(1)\eps^{-n+1}$ includes the \textcolor{red}{domain-size} effect. To the best of the author knowledge, 
theose asymptotics described in Theorem~\ref{new-cor} have not been obtained explicitly in
other literatures about the sinh-Gordon equation~(\ref{eq2})--(\ref{bd2}).

Here we give an application as follows.

\begin{example}\label{ex212-1}
We establish asymptotics of ${u}(r_{\eps})$, ${u}'(r_{\eps})$
and $r_{\eps}$ as $\eps$ tends to zero, where  
\begin{equation}\label{apply-0216}
|u(R)-{u}(r_{\eps})|=\frac{1}{2}|u(R)-b|
\end{equation}
and $b=\displaystyle\lim_{\eps\downarrow0}u(R)$ (cf. Proposition~\ref{prop-ceps}).
Firstly, by (\ref{mainth1-id6}) and (\ref{apply-0216}), 
it yields that
\begin{align}\label{mainth1-ex1}
\langg{u}(r_{\eps})\rangg=b+\frac{\eps}{R}\cdot\frac{\gamma\left(N\cosh^2\frac{b}{2}-1\right)\tanh\frac{b}{4}}{\gamma\cosh\frac{b}{2}+1},
\end{align}
which shares the same leading order term with $u(R)$,
and $|u(r_{\eps})-u(R)|$ is merely of the order $\eps$.
Hence, by the comparison of (\ref{u-0d}) and (\ref{mainth1-ex1}), 
 one may obtain $k(\mathtt{p})=b$ and 
$\mathcal{H}^{\gamma;b}_{\mathtt{p};\mathtt{q}}=\frac{\gamma(N\cosh^2\frac{b}{2}-1)\sech^2\frac{b}{4}}{2(\gamma\cosh\frac{b}{2}+1)}$. Along with (\ref{0308-k})--(\ref{mathcal-H}),
it turns out that 
\begin{align*}
\mathtt{p}=0\,\,\mathrm{and}\,\,\mathtt{q}=\frac{\gamma}{4}\left(1+\frac{N-1}{2R}\sech^2\frac{b}{4}\right)\frac{\left(N\cosh^2\frac{b}{2}-1\right)\sech^2\frac{b}{4}}{\gamma\cosh\frac{b}{2}+1}. 
\end{align*}
The conclusion is
 $R-r_{\eps}\sim\eps^2$ with asymptotics
\begin{align}\label{haha021418-1}
\frac{R-r_{\eps}}{\eps^2}=\frac{\gamma}{4R}\left(1+\frac{N-1}{2R}\sech^2\frac{b}{4}\right)\frac{\left(N\cosh^2\frac{b}{2}-1\right)\sech^2\frac{b}{4}}{\gamma\cosh\frac{b}{2}+1}+{o}_{\eps}(1)
\end{align}
and ${u}'(r_{\eps})\sim\eps^{-1}$ with asymptotics
\begin{align}\label{haha021418-2}
\langg{u}'(r_{\eps})\rangg=
\frac{2}{\eps}\sinh\frac{b}{2}-\frac{1}{R}&\left(4N\sinh\frac{b}{2}\sinh^2\frac{b}{4}+2(N-1)\tanh\frac{b}{4}-\frac{\gamma\left(N\cosh^2\frac{b}{2}-1\right)\tanh\frac{b}{4}\cosh\frac{b}{2}}{\gamma\cosh\frac{b}{2}+1}\right).
\end{align}
\end{example}

It seems that (\ref{haha021418-1}) and (\ref{haha021418-2}) are not easy to be obtained via the method of matching asymptotic expansions because it involves the exact second order terms of the asymptotics of the nonlocal coefficient $\pmb{\mathtt{C}}(u)$.

For thin layered solutions of (\ref{eq2})--(\ref{bd2}), we are also interested in its boundary concentration phenomenon.
To see such phenomena, let $\mathcal{F}\in\mathrm{C}_{\mathrm{loc}}^{0,\tau}(\mathbb{R})$, $\tau\in(0,1]$, be a locally  H\"{o}lder (or Lipschitz)
continuous function with exponent $\tau$ that is independent of $\eps$.
Then both ${\eps}^{-1}|\mathcal{F}(\eps{u}'(r))-\mathcal{F}(0)|$ and ${\eps}^{-1}|\mathcal{F}(u(r))-\mathcal{F}(0)|$ blows up asymptotically near the boundary point.
Indeed, by Lemma~\ref{asy-radial-lem}(ii), we obtain that for $r\in[0,R)$
and $0<\eps\ll1$,
\begin{equation}\label{0217-0217}
    \begin{aligned}
|\mathcal{F}(\eps{u}'(r))-\mathcal{F}(0)|&+|\mathcal{F}(u(r))-\mathcal{F}(0)|\\
&\lesssim\eps^{\tau}|u'(r)|^{\tau}+|u(r)|^{\tau}\lesssim{e}^{-\frac{M_1}{\eps}\tau(R-r)}.
\end{aligned}
\end{equation}
In particular,  as $\eps\downarrow0$,
 both ${\eps}^{-1}|\mathcal{F}(\eps{u}'(r))-\mathcal{F}(0)|$ and ${\eps}^{-1}|\mathcal{F}(u(r))-\mathcal{F}(0)|$
are uniformly bounded to $\eps$ in $\mathrm{L}^1([0,R])$, and converge to zero uniformly in any compact subset of $[0,R)$. 
However, by Proposition~\ref{prop-ceps}, 
${\eps}^{-1}|\mathcal{F}(\eps{u}'(R))-\mathcal{F}(0)|$ and ${\eps}^{-1}|\mathcal{F}(u(R))-\mathcal{F}(0)|$ diverge to infinity (note that $b>0$). This also asserts 
the boundary concentration phenomenon of ${\eps}^{-1}|\mathcal{F}(\eps{u}'(r))-\mathcal{F}(0)|$ and ${\eps}^{-1}|\mathcal{F}(u(r))-\mathcal{F}(0)|$.

The following theorem precisely describes their boundary concentration phenomena via Dirac delta functions concentrated at boundary points~(see Definition~\ref{def1}).

\begin{theorem}[Boundary concentration phenomenon]\label{mainthm-1}
Assume again that the same hypotheses as in Theorem~\ref{bdylayer-exist-thm} hold.
Then for $\mathcal{F}\in\mathrm{C}_{\mathrm{loc}}^{0,\tau}(\mathbb{R};\mathbb{R})$ independent of $\eps$, 
as $\eps\downarrow0$, $\frac{\mathcal{F}(\eps{u}')-\mathcal{F}(0)}{\eps}$ 
and $\frac{\mathcal{F}(u)-\mathcal{F}(0)}{\eps}$ have boundary concentration phenomena
described as follows:
\begin{itemize}
\item[\textbf{(I-i)}] If $\int_{0+}^b\frac{\mathcal{F}(2\sinh\frac{t}{2})-\mathcal{F}(0)}{2\sinh\frac{t}{2}}\dtt\neq0$, there holds
\begin{align}
\frac{\mathcal{F}(\eps{u}')-\mathcal{F}(0)}{\eps}\rightharpoonup\left(\int_{0+}^b\frac{\mathcal{F}(2\sinh\frac{t}{2})-\mathcal{F}(0)}{2\sinh\frac{t}{2}}\dtt\right)\delta_R\quad\,weakly\,\,in\,\,\mathrm{C}([0,R];\mathbb{R}).\label{mainth1-id1-20180204}
\end{align}
\item[\textbf{(I-ii)}] If $\int_{0+}^b\frac{\mathcal{F}(t)-\mathcal{F}(0)}{2\sinh\frac{t}{2}}\dtt\neq0$, there holds
\begin{align}
\frac{\mathcal{F}(u)-\mathcal{F}(0)}{\eps}\rightharpoonup\left(\int_{0+}^b\frac{\mathcal{F}(t)-\mathcal{F}(0)}{2\sinh\frac{t}{2}}\dtt\right)\delta_R\quad\,weakly\,\,in\,\,\mathrm{C}([0,R];\mathbb{R}).\label{mainth1-id2-20180204}
\end{align}
\end{itemize}
Moreover, for $r_{\mathtt{p}}^{\eps}\in\pmb{\mathbb{B}_{\partial}^{\eps}}$
with $\displaystyle\lim_{\eps\downarrow0}\frac{R-r_{\mathtt{p}}^{\eps}}{\eps}=\mathtt{p}$, as $\eps\downarrow0$ we have
\begin{itemize}
\item[\textbf{(II-i)}]  If $\int_{k(\mathtt{p})}^{b}\frac{\mathcal{F}(2\sinh\frac{t}{2})}{2\sinh\frac{t}{2}}\dtt\neq0$, there holds
\begin{align}
\frac{\mathcal{F}(\eps{u}')}{\eps}\chi_{[r_{\mathtt{p}}^{\eps},R]}&\rightharpoonup\left(\int_{k(\mathtt{p})}^{b}\frac{\mathcal{F}(2\sinh\frac{t}{2})}{2\sinh\frac{t}{2}}\dtt\right)\delta_R\quad\,weakly\,\,in\,\,\mathrm{C}([0,R];\mathbb{R}),\label{mainth1-id1-20180319}
\end{align}
where characteristic function~$\chi_{[r_{\mathtt{p}}^{\eps},R]}$ 
is defined by $\chi_{[r_{\mathtt{p}}^{\eps},R]}(r)=1$ for $r\in[r_{\mathtt{p}}^{\eps},R]$, and $\chi_{[r_{\mathtt{p}}^{\eps},R]}(r)=0$ for $r\not\in[r_{\mathtt{p}}^{\eps},R]$, and $k(\mathtt{p})\in(0,b]$ is uniquely determined by (\ref{0308-k}).
\item[\textbf{(II-ii)}]  If $\int_{k(\mathtt{p})}^{b}\frac{\mathcal{F}(t)}{2\sinh\frac{t}{2}}\dtt\neq0$, there holds
\begin{align}
\frac{\mathcal{F}(u)}{\eps}\chi_{[r_{\mathtt{p}}^{\eps},R]}&\rightharpoonup\left(\int_{k(\mathtt{p})}^{b}\frac{\mathcal{F}(t)}{2\sinh\frac{t}{2}}\dtt\right)\delta_R\quad\,weakly\,\,in\,\,\mathrm{C}([0,R];\mathbb{R}).\label{mainth1-id2-20180319}
\end{align}
\end{itemize}
\end{theorem}
\begin{remark}
Integrals $\int_{0+}^b\frac{\mathcal{F}(2\sinh\frac{t}{2})-\mathcal{F}(0)}{2\sinh\frac{t}{2}}\dtt$
and $\int_{0+}^b\frac{\mathcal{F}(t)-\mathcal{F}(0)}{2\sinh\frac{t}{2}}\dtt$ are finite
due to the estimate 
$\left|{\mathcal{F}(2\sinh\frac{t}{2})-\mathcal{F}(0)}\right|+\left|{\mathcal{F}(t)-\mathcal{F}(0)}\right|\lesssim{t}^{\tau-1}\sinh\frac{t}{2}$ for $t>0$ and $0<\tau\leq1$, 
which can be checked via the elementary inequality $\sinh\frac{t}{2}\geq\frac{t}{2}$ for $t\geq0$. 
\end{remark}

We will give the proof of Theorem~\ref{mainthm-1} in Section~\ref{sec-concentration}.
The following example describing the boundary concentration phenomena of $\eps({u}')^2$ and $\eps^{-1}u$
 is a direct result of Theorem~\ref{mainthm-1}.

\begin{example}\label{ex212-2}
Both $\eps({u}')^2$ and $\eps^{-1}u$ have boundary concentration phenomena
in the following senses:
\begin{align*}
\eps(u')^2&{\rightharpoonup4\left(\cosh\frac{b}{2}-1\right)\delta_R},\\
\frac{u}{\eps}&\rightharpoonup\left(\int_0^b\frac{t}{2\sinh\frac{t}{2}}\dtt\right)\delta_R,
\end{align*}
weakly in $\mathrm{C}([0,R];\mathbb{R})$ as $\eps\downarrow0$.
\end{example}


\section{The Dirichlet-to-Neumann approach}\label{sec-dtn}
\noindent

 Let $u\in\mathrm{C}^{\infty}((0,R))\cap\mathrm{C}^1([0,R])$ be the unique classical solution of (\ref{eq2})--(\ref{bd2}) (cf. Proposition~\ref{corapp}). Since $a_0$, $\gamma$ and $\pmb{\mathtt{C}}(u)$ are positive, and $\sinh{u}$ is increasing to $u$,
the standard maximum principle immediately implies
\begin{align}\label{max-u}
0\leq{u}(r)\leq{a}_0,\quad\forall{r}\in(0,R].
\end{align} 
In this section, we shall establish a Dirichlet-to-Neumann map at the boundary point $r=R$,
\begin{align}\label{DtoN-map-id0}
\Lambda_{\eps}(u(R))=u'(R),
\end{align}
 in an asymptotic framework involving the \textcolor{red}{the diameter $2R$ of the original domain~$\Omega=B_R$} (and also, the curvature $R^{-1}$) as $0<\eps\ll1$,
which plays a crucial role in the asymptotics of the nonlocal coefficient $\pmb{\mathtt{C}}(u)$ and
 the proof of Theorem~\ref{mainthm-1}. The asymptotics of $\Lambda_{\eps}(u(R))$ 
is depicted as follows.
\begin{theorem}\label{DtoN-map}
Under the same hypotheses as in Proposition~\ref{prop-ceps},
we assume
\begin{align}\label{assume-u34}
\liminf_{\eps\downarrow0}u(R)>0.
\end{align}
Then, as $0<\eps\ll1$ we have
\begin{align}\label{DtoN-map-id1}
\langg\Lambda_{\eps}(u(R))\rangg=&\,\frac{2}{\eps}\langg\sinh\frac{u(R)}{2}\rangg-\frac{2}{R}\tanh\frac{\langg{u}(R)\ranggg}{4}\left(N\cosh^2\frac{\langg{u}(R)\ranggg}{2}-1\right)
\end{align}
and 
\begin{align}\label{0304-map}
\left|{u}'(R)-\langg\Lambda_{\eps}(u(R))\rangg\right|\lesssim\sqrt{\eps}.
\end{align}
\end{theorem}
Applying the Dirichlet-to-Neumann approach~\eqref{DtoN-map-id1}--\eqref{DtoN-map-id1}, we can establish the refined asymptotics for the thin annular layer of $u$ with respect to $\eps\downarrow0$.

\subsection{Proof of Theorem~\ref{DtoN-map}}\label{0304-newpfsec}
\noindent

To prove Theorem~\ref{DtoN-map}, we need some lemmas.
Firstly, we establish crucial interior estimates as follows.
\begin{lemma}[Interior estimates]\label{asy-radial-lem}
Assume $a_0>0$. For $\eps>0$ and $\gamma>0$, let $u$ be the unique classical solution of (\ref{eq2})--(\ref{bd2}). Then 
\begin{itemize}
\item[(i)]\,\,For $\eps>0$ fixed, $u(r)$ and $u'(r)$ are strictly positive and $u''(r)\geq0$ in $(0,R]$.
\item[(ii)]\,\,As $0<\eps\ll1$, there hold
\begin{align}\label{gradesti}
\max\left\{|{u}(r)|,\,{\gamma\eps}|{u}'(r)|\right\}\leq2a_0{e}^{-\frac{1}{8\eps}({\cosh{a}_0})^{-1/2}(R-r)},\quad{r}\in[0,R].
\end{align}
\end{itemize}
\end{lemma}
\begin{proof}
Note that (\ref{max-u}) implies
\begin{align}\label{0303-2018}
\pmb{\mathtt{C}}(u)\geq\left(\cosh{a}_0\right)^{-1}.
\end{align}
Thus by (\ref{eq2}) and (\ref{0303-2018}), we have
\begin{align}\label{eq-newu}
\eps^2\left(u''+\frac{N-1}{r}u'\right)\geq\left(\cosh{a}_0\right)^{-1}u.
\end{align}
Hence, by Proposition 2.1. of \cite{lhl2016}, (\ref{eq2}) is a second order elliptic equation and the solution $u$ satisfies the unique continuation property. 
Now we give the proof of (i). 
Suppose by contradiction that there exists $r_0\in(0,R)$ such that $u'(r_0)=0$. Then,
multiplying (\ref{eq-newu}) by $r^{N-1}$, integrating the expression over $(0,r_0)$
and using $u'(0)=u'(r_0)=0$ immediately give
\begin{align}\notag
\int_0^{r_0}u(r)\drr=0.
\end{align}
Along with (\ref{max-u}) implies $u\equiv0$ in $[0,r_0]$, and then
the unique continuation property shows that $u$ is trivial in $[0,R]$, a contradiction.
Consequently, $u'>0$ in $(0,R]$. Similarly, by (\ref{max-u}) and unique continuation property, we obtain $u>0$ in $(0,R]$.

Differentiating (\ref{eq2}) to $r$ and using (\ref{0303-2018}) and $u$, $u'>0$ in $(0,R)$, we have
\begin{align}\label{u-pulan}
\eps^2\left(u'''+\frac{N-1}{r}u''\right)=&\left[\frac{(N-1)\eps^2}{r^2}+\pmb{\mathtt{C}}(u)\cosh{u}\right]u'
\geq\left(\cosh{a}_0\right)^{-1}u'>0\,\,\mathrm{in}\,\,(0,R).
\end{align}
For $\eps>0$ fixed,
multiplying (\ref{u-pulan}) by $r^{N-1}$, one arrives at 
$\eps^2(r^{N-1}u''(r))'>0$. Hence, for $r\in(0,R)$ we have
\begin{align}\label{sec-der-est}
r^{N-1}u''(r)\geq\,\liminf_{s\downarrow0+}s^{N-1}u''(s)
=\,\liminf_{s\downarrow0+}\left(\frac{\pmb{\mathtt{C}}_{\eps}(u(s))}{\eps^2}s^{N-1}\sinh{u(s)}-(N-1)s^{N-2}u'(s)\right)\geq0.
\end{align}
Here we have used the facts $\sinh{u(s)}\geq0$, $u'(0)=0$ and $N>1$ to verify (\ref{sec-der-est}).
This implies $u''(r)\geq0$ for $r\in(0,R)$, and completes the proof of (i).

Now we want to prove (ii). 
Multiplying (\ref{eq-newu})
by $u$, one may check that
\begin{align}\notag
\eps^2(u^2)''\geq2\left(\eps^2u'^2-\frac{(N-1)\eps^2}{r}uu'+(\cosh{a}_0)^{-1}u^2\right).
\end{align}
In particular, for $r\in[\frac{R}{2},R]$, by (i)
one has $\eps^2u'^2-\frac{(N-1)\eps^2}{r}uu'\geq\eps^2u'^2-\frac{2(N-1)\eps^2}{R}uu'\geq-\frac{(N-1)^2\eps^2}{R^2}u^2$.
This concludes
\begin{align}\label{uto2-1}
\eps^2(u^2)''\geq&\,2\left[-\frac{(N-1)^2\eps^2}{R^2}+(\cosh{a}_0)^{-1}\right]u^2
\geq\,(\cosh{a}_0)^{-1}u^2,\quad\quad\quad\mathrm{as}\,\,0<\eps<\eps^*(R),
\end{align}
where $\eps^*(R)={R}{(N-1)^{-1}(2\cosh{a}_0)^{-1/2}}$. Applying elliptic comparison arguments to (\ref{uto2-1}) and using (\ref{max-u}), we obtain
\begin{align}
0\leq{u}(r)\leq{a}_0\left(e^{-\frac{(\cosh{a_0})^{-1/2}}{2\eps}(r-\frac{R}{2})}+e^{-\frac{(\cosh{a_0})^{-1/2}}{2\eps}(R-r)}\right),\quad\forall\,r\in[\frac{R}{2},R],
\end{align}\label{uto2-2}
as $0<\eps<\eps^*(R)$. As a consequence, 
\begin{itemize}
 \item[\textbf{(a1).}]\,\,When $r\geq\frac{3}{4}R$, i.e., $R-r\leq{r}-\frac{R}{2}$, we have
\begin{equation*}
0\leq{u}(r)\leq2a_0e^{-\frac{(\cosh{a_0})^{-1/2}}{2\eps}(R-r)}.
\end{equation*}
\item[\textbf{(a2).}]\,\,When $r\in[0,\frac{3}{4}R]$, by $u'\geq0$ we have
\begin{equation*}
0\leq{u}(r)\leq{u}(\frac{3R}{4})\leq2a_0e^{-\frac{(\cosh{a_0})^{-1/2}}{8\eps}R}\leq2a_0e^{-\frac{(\cosh{a_0})^{-1/2}}{8\eps}(R-r)}.
\end{equation*}
\end{itemize}
One can conclude from (a1) and (a2) that
\begin{align}\label{uto2-3}
0\leq{u}(r)\leq2a_0e^{-\frac{(\cosh{a_0})^{-1/2}}{8\eps}(R-r)},\quad\forall\,r\in[0,R],
\end{align}
as $0<\eps<\eps^*(R)$.

Now we deal with the estimate of $u'$. 
Multiplying (\ref{u-pulan})
by $u'$ and using $u'\geq0$, one may check that, for $r\in[\frac{R}{2},R]$,
\begin{align}\label{uto2-1-plun}
\eps^2(u'^2)''\geq(\cosh{a}_0)^{-1}u'^2,\quad\mathrm{as}\,\,0<\eps<\eps^*(R).
\end{align}
Hence, following similar argument of (\ref{uto2-1})--(\ref{uto2-3}), we have
\begin{align}\label{uto2-3-plun}
0\leq{u}'(r)\leq2u'(R)e^{-\frac{(\cosh{a_0})^{-1/2}}{8\eps}(R-r)},\quad\forall\,r\in[0,R],
\end{align}
as $0<\eps<\eps^*(R)$. Moreover, by the boundary condition~(\ref{bd2}) and (\ref{max-u})
we have $u'(R)\leq\frac{a_0}{\gamma\eps}$. Along with (\ref{uto2-3-plun}) yields
\begin{align}\label{uto2-4-plun}
0\leq{u}'(r)\leq\frac{2a_0}{\gamma\eps}e^{-\frac{(\cosh{a_0})^{-1/2}}{8\eps}(R-r)},\quad\forall\,r\in[0,R],
\end{align}
as $0<\eps<\eps^*(R)$.
Therefore, by (\ref{uto2-3}) and (\ref{uto2-4-plun}) we get (\ref{gradesti}).

 Therefore,
we complete the proof of Lemma~\ref{asy-radial-lem}.
\end{proof}

\begin{lemma}\label{cruc-in-thm}
Under the same hypotheses as in Proposition~\ref{prop-ceps}, as $0<\eps\ll1$ we have
\begin{align}
\left|\frac{\eps^2}{2}u'^2(t)+(N-1)\eps^2\int_{\frac{R}{2}}^t\frac{1}{r}u'^2(r)\drr-\pmb{\mathtt{C}}(u)\left(\cosh{u(t)}-1\right)\right|&\lesssim{e}^{-\frac{R}{16\eps}(\cosh{a}_0)^{-1/2}},\label{cr-in1}\\
\left|\eps{u}'(t)-2\sqrt{\pmb{\mathtt{C}}(u)}\sinh\frac{u(t)}{2}\right|&\lesssim\sqrt{\eps},\quad\,\,\,t\in(0,R],\label{cr-in2-add}
\end{align}
and
\begin{align}\label{cr-in2}
\left|\pmb{\mathtt{C}}(u)-1+\eps^2\int_{\frac{R}{2}}^R\left(\frac{N-1}{r}-\frac{N-2}{2R^N}r^{N-1}\right)u'^2(r)\drr\right|\lesssim{e}^{-\frac{R}{16\eps}(\cosh{a}_0)^{-1/2}}.
\end{align}
\end{lemma}
\begin{proof}
Putting $t=\frac{R}{2}$ into (\ref{cr-in3}) and using (\ref{gradesti}) and (\ref{0303-2018}), one may check that
\begin{align}\label{cr-in5}
|\pmb{\mathtt{C}}(u)+\mathtt{K}_{\eps}|\leq&\pmb{\mathtt{C}}(u)\left|\cosh{u}(\frac{R}{2})-1\right|
+\frac{\eps^2}{2}u'^2(\frac{R}{2})
\lesssim\,u(\frac{R}{2})+{\eps^2}u'^2(\frac{R}{2})\lesssim{e}^{-\frac{R}{8\eps}(\cosh{a}_0)^{-1/2}},\,\,\mathrm{as}\,\,0<\eps\ll1
\end{align}
Along with (\ref{cr-in3}) immediately yields (\ref{cr-in1}).

Moreover, by (\ref{gradesti}) and (\ref{cr-in1}), one may check that
\begin{align}\label{cr-in6-add}
\big|{\eps^2}u'^2(t)-&\,2\pmb{\mathtt{C}}(u)\left(\cosh{u(t)}-1\right)\big|
\lesssim\,\eps^2\int_{\frac{R}{2}}^t\frac{1}{r}u'^2(r)\drr+{e}^{-\frac{R}{8\eps}(\cosh{a}_0)^{-1/2}}\lesssim\eps.
\end{align}
On the other hand, since $u\geq0$ and $u'\geq0$, one finds
\begin{align}
 \big|{\eps^2}u'^2(t)-&\,2\pmb{\mathtt{C}}(u)\left(\cosh{u(t)}-1\right)\big|
\geq\left({\eps}u'(t)-\sqrt{2\pmb{\mathtt{C}}(u)}\sinh\frac{u(t)}{2}\right)^2.
\end{align}
Along with (\ref{cr-in6-add}), we get (\ref{cr-in2-add}).

It remains to prove (\ref{cr-in2}). Multiplying (\ref{cr-in3}) by $t^{N-1}$ and integrating the result over $(0,R)$,
we have
\begin{align}\label{cr-in6}
\frac{\eps^2}{2}\int_0^Ru'^2(t)t^{N-1}\dtt+(N-1)\eps^2\int_0^Rt^{N-1}\int_{\frac{R}{2}}^t\frac{1}{r}u'^2(r)\drr\dtt=\frac{R^N}{N}\left(1+\mathtt{K}_{\eps}\right).
\end{align}
On the other hand, using the integration by parts we have
\begin{align}\label{cr-in7}
\frac{\eps^2}{2}\int_0^Ru'^2(t)t^{N-1}\dtt+&(N-1)\eps^2\int_0^Rt^{N-1}\int_{\frac{R}{2}}^t\frac{1}{r}u'^2(r)\drr\dtt\notag\\
&\quad\quad=\frac{\eps^2}{N}\int_{\frac{R}{2}}^R\left((N-1)\frac{R^N}{t}-\frac{N-2}{2}t^{N-1}\right)u'^2(t)\dtt\\
&\quad\quad\,\,\quad\quad-\frac{N-2}{2N}\eps^2\int_0^{\frac{R}{2}}t^{N-1}u'^2(t)\dtt.\notag
\end{align}
 Combining (\ref{cr-in5})--(\ref{cr-in7})
and using the gradient estimate in (\ref{gradesti}), it follows
\begin{align}\label{cr-in8}
&\left|\pmb{\mathtt{C}}(u)-1+\eps^2\int_{\frac{R}{2}}^R\left(\frac{N-1}{t}-\frac{N-2}{2R^N}t^{N-1}\right)u'^2(t)\dtt\right|\notag\\
&\quad\leq|\pmb{\mathtt{C}}(u)+\mathtt{K}_{\eps}|+\frac{N-2}{2R^N}\eps^2\int_0^{\frac{R}{2}}t^{N-1}u'^2(t)\dtt
\lesssim{e}^{-\frac{R}{16\eps}(\cosh{a}_0)^{-1/2}},\,\,\mathrm{as}\,\,0<\eps\ll1.
\end{align}
This proves (\ref{cr-in2}) and completes the proof of Lemma~\ref{cruc-in-thm}.
\end{proof}

\begin{lemma}\label{dirac-lem}
Under the same hypotheses as in Proposition~\ref{prop-ceps}, as $0<\eps\ll1$ we have
\begin{align}\label{dirac-1}
&\left|\eps\int_0^Rg(r)u'^2(r)\drr-4g(R)\left(\cosh\frac{u(R)}{2}-1\right)\right|
\lesssim\max_{[R-\sqrt{\eps},R]}|g(r)-g(R)|+\eps
\stackrel{\mathrm{as}\,\,\eps\downarrow0}{-\!\!\!-\!\!\!-\!\!\!-\!\!\!\rightarrow}0,
\end{align}
where $g\in\mathrm{C}([0,R])$ is a continuous function independent of $\eps$. Moreover, there holds
\begin{align}\label{ce-12-3}
\left|\pmb{\mathtt{C}}(u)-1+\frac{2N}{R}\left(\cosh\frac{u(R)}{2}-1\right)\eps\right|\lesssim\eps^{3/2}.
\end{align}
\end{lemma}
\begin{proof}
We write the integral in (\ref{dirac-1}) as 
\begin{align}\label{dirac-2}
\eps\int_0^Rg(r)u'^2(r)\drr=\eps\left(\int_0^{R-\sqrt{\eps}}+\int_{R-\sqrt{\eps}}^R\right)g(r)u'^2(r)\drr,\quad0<\eps\ll1.
\end{align}
Note that $\sup_{[0,R]}|g|$ is finite and indepentent of $\eps$.
Thus, using (\ref{gradesti}) and passing through simple calculations, one finds
\begin{align}\label{dirac-3}
\eps\left|\int_0^{R-\sqrt{\eps}}g(r)u'^2(r)\drr\right|\lesssim{e}^{-\frac{1}{4\sqrt{\eps}}(\cosh{a}_0)^{-1/2}}.
\end{align}
On the other hand, by (\ref{max-u}), Lemma~\ref{asy-radial-lem}(i) and (\ref{cr-in2-add}), we have
\begin{align}\label{dirac-5}
\Bigg|\eps\int_{R-\sqrt{\eps}}^Rg(r){u}'^2(r)\drr-&2\sqrt{\pmb{\mathtt{C}}(u)}\int_{R-\sqrt{\eps}}^Rg(r)\sinh\frac{u(r)}{2}u'(r)\drr\Bigg|\notag\\
=&\,\int_{R-\sqrt{\eps}}^Ru'(r)|g(r)|\left|\eps{u}'(r)-2\sqrt{\pmb{\mathtt{C}}(u)}\sinh\frac{u(r)}{2}\right|\drr\\
\lesssim&\,\sqrt{\eps}\int_{R-\sqrt{\eps}}^Ru'(r)\drr=\sqrt{\eps}(u(R)-u(R-\sqrt{\eps}))\lesssim\sqrt{\eps}.\notag
\end{align}
We shall further estimate the term $2\sqrt{\pmb{\mathtt{C}}(u)}\int_{R-\sqrt{\eps}}^Rg(r)\sinh\frac{u(r)}{2}u'(r)\drr$ for $0<\eps\ll1$.
Note first that (\ref{gradesti}) and (\ref{cr-in2}) imply $\pmb{\mathtt{C}}(u)\to1$ as $\eps\downarrow0$.
Using the identity
\begin{align}\label{dirac-5ad}
\sqrt{\pmb{\mathtt{C}}(u)}\int_{R-\sqrt{\eps}}^R&g(r)\sinh\frac{u(r)}{2}u'(r)\drr\notag\\
=&\int_{R-\sqrt{\eps}}^R\left[\sqrt{\pmb{\mathtt{C}}(u)}\left(g(r)-g(R)\right)+\left(\sqrt{\pmb{\mathtt{C}}(u)}-1\right)g(R)\right]\sinh\frac{u(r)}{2}u'(r)\drr\\
&+2g(R)\left(\cosh\frac{u(R)}{2}-\cosh\frac{u(R-\sqrt{\eps})}{2}\right),\notag
\end{align} 
we can arrive at
\begin{align}\label{dirac-6}
&\left|\sqrt{\pmb{\mathtt{C}}(u)}\int_{R-\sqrt{\eps}}^Rg(r)\sinh\frac{u(r)}{2}u'(r)\drr-2g(R)\left(\cosh\frac{u(R)}{2}-1\right)\right|\notag\\
&\hspace{36pt}\leq4\left(\sqrt{\pmb{\mathtt{C}}(u)}\max_{[R-\sqrt{\eps},R]}|g(r)-g(R)|+\left|\left(\sqrt{\pmb{\mathtt{C}}(u)}-1\right)g(R)\right|\right)\sinh\frac{u(R)}{2}\\
&\hspace{50pt}+2|g(R)|\left(\cosh\frac{u(R-\sqrt{\eps})}{2}-1\right)\notag\\
&\hspace{36pt}\lesssim\max_{[R-\sqrt{\eps},R]}|g(r)-g(R)|+\left|\sqrt{\pmb{\mathtt{C}}(u)}-1\right|+{e}^{-\frac{(\cosh{a}_0)^{-1/2}}{4\sqrt{\eps}}}.\notag
\end{align}
Here we have used (\ref{max-u}) and (\ref{gradesti}) to
get $|\cosh\frac{u(R-\sqrt{\eps})}{2}-1|\lesssim{u}^2(R-\sqrt{\eps})\lesssim\exp\left(-\frac{1}{4\sqrt{\eps\cosh{a}_0}}\right)$
which asserts the last inequality of (\ref{dirac-6}).
Since $g$ is continuous and independent of $\eps$ and $R-\sqrt{\eps}\to{R}$ as $\eps\downarrow0$,
by (\ref{gradesti}), (\ref{cr-in2}), (\ref{dirac-2})--(\ref{dirac-5}) and (\ref{dirac-6}),
we get
\begin{align}\label{dirac-1-new0303}
\eps\int_0^Rg(r)u'^2(r)\drr=4g(R)\left(\cosh\frac{u(R)}{2}-1\right)+o_{\eps}(1).
\end{align}
Furthermore, when we set a function $g\in\mathrm{C}([0,R])$ satisfying
\begin{align}\notag
g(r)=0\,\,\mathrm{for}\,\,r\in[0,\frac{R}{4}];\,\,g(r)=\frac{N-1}{r}-\frac{N-2}{2R^N}r^{N-1}\,\,\mathrm{for}\,\,r\in[\frac{R}{2},R],
\end{align}
 by 
 (\ref{gradesti}), (\ref{cr-in2}) and (\ref{dirac-1-new0303}), we have
\begin{align}\label{ce0304-2018}
\pmb{\mathtt{C}}(u)=&1-\eps^2\int_{\frac{R}{2}}^R\left(\frac{N-1}{r}-\frac{N-2}{2R^N}r^{N-1}\right)u'^2(r)\drr+o_{\eps}(1)\notag\\
=&1-4{g}(R)\left(\cosh\frac{u(R)}{2}-1+o_{\eps}(1)\right)\eps\\
=&1-\frac{2N}{R}\left(\cosh\frac{u(R)}{2}-1+o_{\eps}(1)\right)\eps.\notag
\end{align}
In particular, $\left|\sqrt{\pmb{\mathtt{C}}(u)}-1\right|\lesssim\eps$ as $0<\eps\ll1$.
Along with (\ref{dirac-6}), we arrive at (\ref{dirac-1}).
Moreover, by (\ref{cr-in2}) and (\ref{dirac-1}), (\ref{ce0304-2018}) can be improved by
\begin{align}
&\left|\pmb{\mathtt{C}}(u)-1+\frac{2N}{R}\left(\cosh\frac{u(R)}{2}-1\right)\eps\right|\notag\\
&\quad\leq
\left|\pmb{\mathtt{C}}(u)-1+\eps^2\int_{\frac{R}{2}}^R\left(\frac{N-1}{r}-\frac{N-2}{2R^N}r^{N-1}\right)u'^2(r)\drr\right|\notag\\
&\quad\quad+\eps\left|\eps\int_{\frac{R}{2}}^R\left(\frac{N-1}{r}-\frac{N-2}{2R^N}r^{N-1}\right)u'^2(r)\drr-\frac{2N}{R}\left(\cosh\frac{u(R)}{2}-1\right)\right|\notag\\
&\quad\lesssim{e}^{-\frac{R}{16\eps}(\cosh{a}_0)^{-1/2}}+\eps^{3/2}
\lesssim\eps^{3/2}.\notag
\end{align}
Therefore, we prove (\ref{ce-12-3}) and completes the proof of Lemma~\ref{dirac-lem}.
\end{proof}
\begin{remark}\label{rk-1}
Assume $g\in\mathrm{C}^{\alpha}([0,R])$ is H\"{o}lder continuous with exponent $\alpha\in(0,1)$.
Owing to (\ref{ce-12-3}) and (\ref{dirac-6}), (\ref{dirac-1}) can be improved by
\begin{align}\notag
\eps\int_0^Rg(r)u'^2(r)\drr=4g(R)\left(\cosh\frac{u(R)}{2}-1\right)+\eps^{\alpha/2}\mathcal{O}(1).
\end{align} 
\end{remark}

Using Lemmas \ref{cruc-in-thm} and \ref{dirac-lem},
we now give the proof of Theorem~\ref{DtoN-map} as follows. 
\begin{proof}[\textbf{Proof of Theorem~\ref{DtoN-map}}]
Setting $t=R$ in (\ref{cr-in1}) gives
\begin{align*}
\left|\frac{\eps^2}{2}u'^2(R)+(N-1)\eps^2\int_{\frac{R}{2}}^R\frac{1}{r}u'^2(r)\drr
-\pmb{\mathtt{C}}(u)(\cosh{u(R)}-1)\right|\lesssim{e}^{-\frac{R}{16\eps}(\cosh{a}_0)^{-1/2}}.
\end{align*}
Next, consider a continuous function $g$ with $g(r)=\frac{1}{r}$ for $r\in[\frac{R}{2},R]$
and $g(r)=0$ near $r=0$ in (\ref{dirac-1}). Using (\ref{gradesti}), one immediately finds
$\Big|\eps^2\int_{\frac{R}{2}}^R\frac{1}{r}u'^2(r)\drr-\frac{8}{R}\eps\sinh^2\frac{u(R)}{4}\Big|\lesssim{\eps}^{3/2}$.
Note also the estimate of $\pmb{\mathtt{C}}(u)$ in (\ref{ce-12-3}).
As a consequence, after making appropriate manipulations we obtain
\begin{align}
&\left|\frac{\eps^2}{2}u'^2(R)-2\sinh^2\frac{u(R)}{2}\left[1-\frac{4}{R}\cdot\frac{\sinh^2\frac{u(R)}{4}}{\sinh^2\frac{u(R)}{2}}\left(N\cosh^2\frac{u(R)}{2}-1\right)\eps\right]\right|\notag\\
&\hspace{36pt}\leq\left|\frac{\eps^2}{2}u'^2(R)+(N-1)\eps^2\int_{\frac{R}{2}}^R\frac{1}{r}u'^2(r)\drr
-\pmb{\mathtt{C}}(u)(\cosh{u(R)}-1)\right|\notag\\
&\hspace{36pt}\quad+(N-1)\left|\eps^2\int_{\frac{R}{2}}^R\frac{1}{r}u'^2(r)\drr-\frac{8}{R}\eps\sinh^2\frac{u(R)}{4}\right|\notag\\
&\hspace{36pt}\quad+\left|\pmb{\mathtt{C}}(u)(\cosh{u(R)}-1)-2\sinh^2\frac{u(R)}{2}\left(1-\frac{4N}{R}\eps\sinh^2\frac{u(R)}{4}\right)\right|
\lesssim\,\eps^{3/2}.\notag
\end{align}
Since $u(R)$ is uniformly bounded to $\eps$ (cf. Lemma~\ref{asy-radial-lem}(i)), together with 
assumption~(\ref{assume-u34}) we conclude
\begin{align}
\left|{u}'(R)-2\sinh\frac{u(R)}{2}\left[\frac{1}{\eps}-\frac{2}{R}\cdot\frac{\sinh^2\frac{u(R)}{4}}{\sinh^2\frac{u(R)}{2}}\left(N\cosh^2\frac{u(R)}{2}-1\right)\right]\right|\lesssim{\eps}^{1/2},\notag
\end{align}
together with (\ref{DtoN-map-id0}), we get (\ref{DtoN-map-id1}) and (\ref{0304-map})
and complete the proof of Theorem~\ref{DtoN-map}.~\end{proof}

\subsection{Proof of Proposition~\ref{prop-ceps}}\label{sec03162018}
\noindent

The proof of Proposition~\ref{prop-ceps} is stated as follows.

By the Robin boundary condition~(\ref{bd2}) and Theorem~\ref{DtoN-map},
we have
\begin{align}\label{conph-1}
\left|\frac{u(R)+2\gamma\sinh\frac{u(R)}{2}}{\eps}-\frac{2\gamma}{R}{\tanh\frac{u(R)}{4}}\left(N\cosh^2\frac{u(R)}{2}-1\right)-a_0\right|\lesssim\sqrt{\eps},
\end{align}
as $0<\eps\ll1$. Then by (\ref{arbequ}) and the leading term of \eqref{conph-1} with respect to $\eps$ it is apparent that 
\begin{align}\label{add-2020}
\left|(u(R)-b)+2\gamma\left(\sinh\frac{u(R)}{2}-\sinh\frac{b}{2}\right)\right|\lesssim{\eps}.
\end{align}
Since $s+2\gamma\sinh\frac{s}{2}$ is increasing to $s$, by (\ref{max-u}) and \eqref{add-2020} it immediately follows
\begin{align}\label{conph-2}
\left|u(R)-b\right|\lesssim{\eps}.
\end{align}
Consequently, by (\ref{max-u}), (\ref{ce-12-3}) and (\ref{conph-2}), we obtain
\begin{align*}
\left|\pmb{\mathtt{C}}(u)-1+\frac{2N}{R}\left(\cosh\frac{b}{2}-1\right)\eps\right|\lesssim&\,\eps^{3/2}+\left|\cosh\frac{b}{2}-\cosh\frac{u(R)}{2}\right|\eps\\
\lesssim&\,\eps^{3/2}+\left|b-u(R)\right|\eps\lesssim\eps^{3/2}.
\end{align*}
This gives (\ref{ceps-1}) and 
\begin{equation}\label{11-0305-1230}
\left|\pmb{\mathtt{C}}(u)-\langg\pmb{\mathtt{C}}(u)\rangg\right|\lesssim\eps^{3/2}.
\end{equation}
Now we shall prove (\ref{mainth1-id6}) and (\ref{mainth1-id5}).
By (\ref{conph-2}), we set
\begin{align}\label{dirich-1119}
\begin{cases}
u(R)=b+\widetilde{b}_{\eps},\\
b\,\,\mathrm{is\,\,uniquely\,\,determined\,\,by\,\,(\ref{arbequ}),\,\,and}\,\,|\widetilde{b}_{\eps}|\lesssim\eps\,\,\mathrm{as}\,\,\eps\downarrow0.
\end{cases}
\end{align}
It suffices to calculate the leading order term and an optimal second order error of $\widetilde{b}_{\eps}$ with respect to $\eps$ as $0<\eps\ll1$. Notice the relation $b+2\gamma\sinh\frac{b}{2}=a_0$.
Putting (\ref{dirich-1119}) into (\ref{conph-1})
and passing through simple calculations, one arrives at
\begin{align}\label{dirich-1119-1043}
\left|\widetilde{b}_{\eps}-\eps\frac{2\gamma}{R}\frac{\left(N\cosh^2\frac{b}{2}-1\right)\tanh\frac{b}{4}}{\gamma\cosh\frac{b}{2}+1}\right|\lesssim\eps^{3/2}.
\end{align}
Here we have used (\ref{arbequ}) to get approximations
\begin{align}
\bigg|\sinh\frac{b+\widetilde{b}_{\eps}}{2}-&\,\sinh\frac{b}{2}-\frac{\widetilde{b}_{\eps}}{2}\cosh\frac{b}{2}\bigg|
\lesssim\eps^2,\label{0308-a1}\\
\bigg|\tanh\frac{b+\widetilde{b}_{\eps}}{4}-&\,\tanh\frac{b}{4}-\frac{\widetilde{b}_{\eps}}{4}\sech^2\frac{b}{4}\bigg|
\lesssim\eps^2,\label{0308-a2}\\
\bigg|\cosh^2\frac{b+\widetilde{b}_{\eps}}{2}-&\,\cosh^2\frac{b}{2}-\frac{\widetilde{b}_{\eps}}{2}\sinh{b}\bigg|\lesssim\eps^2,\label{0308-a3}
\end{align}
and
\begin{align*}
\bigg|{\tanh\frac{b+\widetilde{b}_{\eps}}{4}}\left(N\cosh^2\frac{b+\widetilde{b}_{\eps}}{2}-1\right)-\tanh\frac{b}{4}\left(N\cosh^2\frac{b}{2}-1\right)\bigg|\lesssim\eps.
\end{align*}
As a consequence,
\begin{align}\label{dirich-1119-1004}
\left|u(R)-b-\eps\frac{2\gamma}{R}\frac{\left(N\cosh^2\frac{b}{2}-1\right)\tanh\frac{b}{4}}{\gamma\cosh\frac{b}{2}+1}\right|\lesssim\eps^{3/2}.
\end{align}
Along with the Robin boundary condition~(\ref{bd2}) yields
\begin{align*}
\left|u'(R)-\frac{2}{\eps}\sinh\frac{b}{2}+\frac{2}{R}\cdot\frac{\left(N\cosh^2\frac{b}{2}-1\right)\tanh\frac{b}{4}}{\gamma\cosh\frac{b}{2}+1}\right|\lesssim\eps^{1/2}.
\end{align*}
Therefore, we get (\ref{mainth1-id6}) and (\ref{mainth1-id5})
and 
\begin{align}\label{12470305}
\left|u(R)-\langg{u}(R)\rangg\right|\lesssim\eps^{3/2},\quad\left|u'(R)-\langg{u}'(R)\rangg\right|\lesssim\eps^{1/2}.
\end{align}

(\ref{11-0305}) immediately follows from (\ref{11-0305-1230}) and (\ref{12470305}), and, therefore,
 the proof of Proposition~\ref{prop-ceps} is completed.

\section{\textcolor{red}{Domain-size} effects on the thin layer with pointwise asymptotics}\label{sec-curvature}

\subsection{Proof of Theorem~\ref{bdylayer-exist-thm}}\label{sec-pf-0315-2018}
\noindent

Recall $0<u(r)\leq{b}$ and $u'(r)>0$ in $(0,R]$ (cf. Lemma~\ref{asy-radial-lem}(i)).
To prove Theorem~\ref{bdylayer-exist-thm}, we need the following estimate.

\textbf{Claim~1.} For $r_{\eps}\in\pmb{\mathbb{B}_{\partial}^{\eps}}$, as $0<\eps\ll1$ there holds
\begin{align}\label{1120-2}
\left|\log\left|\frac{\tanh\frac{u(R)}{4}}{\tanh\frac{u(r_{\eps})}{4}}\right|-\sqrt{\pmb{\mathtt{C}}(u)}\cdot\frac{R-r_{\eps}}{\eps}\right|\lesssim\frac{\sqrt{\eps}}{\sinh\frac{u(r_{\eps})}{2}}\cdot\sup_{0<\eps\ll1}\frac{R-r_{\eps}}{\eps}.
\end{align}
Moreover,
\begin{align}\label{claim0315}
\mathrm{if}\,\,\limsup_{\eps\downarrow0}\frac{R-r_{\eps}}{\eps}<\infty\,\,\mathrm{and}\,\,\lim_{\eps\downarrow0}u(r_{\eps})=0,\,\,\mathrm{then}\,\,
\lim_{\eps\downarrow0}\frac{1}{\sqrt{\eps}}\sinh\frac{u(r_{\eps})}{2}=0.
\end{align}
\begin{proof}[Proof of Claim~1]
By (\ref{max-u}) and (\ref{cr-in2-add}) we have
\begin{align}\label{1120-1}
\left|\frac{\eps{u}'(t)}{\sinh\frac{u(t)}{2}}-2\sqrt{\pmb{\mathtt{C}}(u)}\right|\lesssim\frac{\sqrt{\eps}}{\sinh\frac{u(r_{\eps})}{2}},
\end{align}
for $t\in[r_{\eps},R]$. Integrating (\ref{1120-1}) over $[r_{\eps},R]$ 
and using
\begin{align}\label{chchchha-5}
\int_{r_{\eps}}^R\frac{{u}'(r)}{\sinh\frac{u(t)}{2}}\dtt=2\log\left|\frac{\tanh\frac{u(R)}{4}}{\tanh\frac{u(r_{\eps})}{4}}\right|,
\end{align}
one immediately obtains (\ref{1120-2}).

Next, we assume that $r_{\eps}\in\pmb{\mathbb{B}_{\partial}^{\eps}}$ satisfies $\displaystyle\lim_{\eps\downarrow0}u(r_{\eps})=0$.
Then by (\ref{mainth1-id6}) and $b>0$, we have
$\lim_{\eps\downarrow0}\log\left|\frac{\tanh\frac{u(R)}{4}}{\tanh\frac{u(r_{\eps})}{4}}\right|=\infty$,
together with (\ref{ceps-1}) and (\ref{1120-2}), we conclude 
\begin{equation*}
0\leq\limsup_{\eps\downarrow0}\frac{1}{\sqrt{\eps}}\sinh\frac{u(r_{\eps})}{2}\lesssim\limsup_{\eps\downarrow0}\bigg(\log\left|\frac{\tanh\frac{u(R)}{4}}{\tanh\frac{u(r_{\eps})}{4}}\right|-\sqrt{\pmb{\mathtt{C}}(u)}\cdot\frac{R-r_{\eps}}{\eps}\bigg)^{-1}=0.
\end{equation*} 
This proves (\ref{claim0315}) and completes the proof of Claim~1. 
\end{proof}

Before proving Theorem~\ref{bdylayer-exist-thm}, we notice that by 
(\ref{ceps-1}) and (\ref{cr-in2-add}), there must hold
\begin{align}\label{steven}
\liminf_{\eps\downarrow0}{u}(r_{\eps})>0\,\Longleftrightarrow\,\liminf_{\eps\downarrow0}\eps{u}'(r_{\eps})>0.
\end{align}
We are now turning to the proof of Theorem~\ref{bdylayer-exist-thm}.
 \begin{proof}[Proof of Theorem~\ref{bdylayer-exist-thm}]
By Lemma~\ref{asy-radial-lem}(ii) and (\ref{steven}), it suffices to prove $\liminf_{\eps\downarrow0}u(r_{\eps})>0$ for $r_{\eps}\in\pmb{\mathbb{B}_{\partial}^{\eps}}$ satisfying $\limsup_{\eps\downarrow0}\frac{R-r_{\eps}}{\eps}<\infty$.
Suppose by contradiction that there exists $\hat{r}_{\eps}\in\pmb{\mathbb{B}_{\partial}^{\eps}}$ with $\frac{R-\hat{r}_{\eps}}{\eps}=\hat{\mathtt{p}}+o_{\eps}(1)$ such that $\lim_{\eps\downarrow0}u(\hat{r}_{\eps})=0$, where $\hat{\mathtt{p}}\geq0$
is independent of $\eps$. Then by (\ref{claim0315}) we have
\begin{align}\label{0315-ni1035}
\lim_{\eps\downarrow0}\frac{1}{\sqrt{\eps}}\sinh\frac{u(\hat{r}_{\eps})}{2}=0.
\end{align}
On the other hand, we set
\begin{equation}\label{0316-mor}
 \widetilde{c}_{\eps}=2\log(\sqrt{\eps}+\sqrt{\eps+1}).
\end{equation}
By (\ref{mainth1-id6}), (\ref{0315-ni1035}) and (\ref{0316-mor}) we have
 $u(\hat{r}_{\eps})<\widetilde{c}_{\eps}<u(R)$ as $0<\eps\ll1$ due to
\begin{align*}
0\stackrel{\mathrm{as}\,\,\eps\downarrow0}{\leftarrow\!\!\!-\!\!\!-\!\!\!-\!\!\!-}\frac{1}{\sqrt{\eps}}\sinh\frac{u(\hat{r}_{\eps})}{2}<\frac{1}{\sqrt{\eps}}\sinh\frac{\widetilde{c}_{\eps}}{2}=1<\frac{1}{\sqrt{\eps}}\sinh\frac{u(R)}{2}\stackrel{\mathrm{as}\,\,\eps\downarrow0}{-\!\!\!-\!\!\!-\!\!\!-\!\!\!\rightarrow}\infty
\end{align*}
and the fact that $\sinh{s}$
is a strictly increasing function.
By the intermediate value theorem there exists $\widetilde{r}_{\eps}\in[\hat{r}_{\eps},R]\subset\pmb{\mathbb{B}_{\partial}^{\eps}}$
such that $u(\widetilde{r}_{\eps})=\widetilde{c}_{\eps}$.
In particular, $\widetilde{r}_{\eps}\in\pmb{\mathbb{B}_{\partial}^{\eps}}$ satisfies
\begin{align*}
\limsup_{\eps\downarrow0}&\frac{R-\widetilde{r}_{\eps}}{\eps}\leq\hat{\mathtt{p}},\,\,
 \lim_{\eps\downarrow0}u(\widetilde{r}_{\eps})=0,\,\,\mathrm{and},\,\,
\lim_{\eps\downarrow0}\frac{1}{\sqrt{\eps}}\sinh\frac{u(\widetilde{r}_{\eps})}{2}\neq0,
\end{align*}
contradicting to (\ref{claim0315}).
As a consequence, 
\begin{align}\label{0315hence}
\liminf_{\eps\downarrow0}u(r_{\eps})>0\,\,\mathrm{for\,\,each}\,\,r_{\eps}\in\pmb{\mathbb{B}_{\partial}^{\eps}}.
\end{align}
Along with (\ref{steven}),
therefore, we get $\liminf_{\eps\downarrow0}{\eps}u'(r_{\eps})>0$ and complete the proof of Theorem~\ref{bdylayer-exist-thm}.
\end{proof}

\subsection{Proof of Theorem~\ref{new-cor}}\label{sec-mainthm-2}
\noindent

To prove Theorem~\ref{new-cor}, we need to collect some preliminary estimates.
Firstly, based on Theorem~\ref{DtoN-map},
we shall generalize the concept of the Dirichlet-to-Neumann map at $r_{\eps}\in\pmb{\mathbb{B}_{\partial}^{\eps}}$,
and establish more refined asymptotic approximations of $u'(r_{\eps})$
and $(R-r_{\eps})/\eps$ as $0<\eps\ll1$.

\begin{lemma}\label{lem-1119-9011}
Let $r_{\eps}\in\pmb{\mathbb{B}_{\partial}^{\eps}}$.
Under the same hypotheses as in Theorem~\ref{bdylayer-exist-thm},
 as $0<\eps\ll1$ we have 
\begin{align}\label{1119-9011-id}
\left|u'(r_{\eps})-2\sinh\frac{u(r_{\eps})}{2}\left[\frac{1}{\eps}-\frac{1}{R}\left(2N\sinh^2\frac{b}{4}+\frac{N-1}{2}\sech^2\frac{u(r_{\eps})}{4}\right)\right]\right|\lesssim\sqrt{\eps}
\end{align}
and
\begin{equation}\label{new-r-eps0305}
\begin{aligned}
\Bigg|\frac{R-r_{\eps}}{\eps^2}-\left(\frac{1}{\eps}+\frac{2N}{R}\sinh^2\frac{b}{4}\right)
\Bigg[&\left(1+\frac{N-1}{2R}\right)\log\left|\frac{\tanh\frac{u(R)}{4}}{\tanh\frac{u(r_{\eps})}{4}}\right|\\
&\quad+\frac{N-1}{4R}\left(\tanh^2\frac{u(r_{\eps})}{4}-\tanh^2\frac{u(R)}{4}\right)\Bigg]\Bigg|\lesssim\sqrt{\eps}.
\end{aligned}
\end{equation}
\end{lemma}
\begin{proof}
 We first deal with (\ref{1119-9011-id}) via the concept of Proposition~\ref{prop-ceps} and Lemma~\ref{dirac-lem}.
Note that $r_{\eps}\in\pmb{\mathbb{B}_{\partial}^{\eps}}$ implies $\left|\frac{1}{r_{\eps}}-\frac{1}{R}\right|\lesssim\eps$ and $R-\sqrt{\eps}<r_{\eps}\leq{R}$ as $0<\eps\ll1$.
Hence, for function $g\in\mathrm{C}([\frac{R}{2},R])$
we may follow the same argument of (\ref{dirac-2})--(\ref{dirac-6}) to get
\begin{equation}\label{ccc-hahaha}
    \begin{aligned}
\eps\int_{\frac{R}{2}}^{r_{\eps}}g(r)u'^2(r)\drr=&\,\eps\left\{\int_{\frac{R}{2}}^{R-\sqrt{\eps}}+\int_{R-\sqrt{\eps}}^{r_{\eps}}\right\}g(r)u'^2(r)\drr\\
=&\,8g(R)\sinh^2\frac{u(r_{\eps})}{4}+\pi_{\eps;g}(r_{\eps})
\end{aligned}
\end{equation}
with
\begin{align}\label{pi-eps}
|\pi_{\eps;g}(r_{\eps})|\lesssim\max_{r\in[R-\sqrt{\eps},r_{\eps}]}|g(r)-g(R)|+\eps\stackrel{\mathrm{as}\,\,\eps\downarrow0}{-\!\!\!-\!\!\!-\!\!\!-\!\!\!\rightarrow}0.
\end{align}
Since $\displaystyle\liminf_{\eps\downarrow0}u(r_{\eps})>0$ (by Theorem~\ref{bdylayer-exist-thm}),
using (\ref{cr-in1}), (\ref{ccc-hahaha}), Proposition~\ref{prop-ceps} and Lemma~\ref{asy-radial-lem}(i)
we can derive a relationship between $u(r_{\eps})$ and $u'(r_{\eps})$ as follows:
\begin{align}\label{chchchha}
\eps{u}'(r_{\eps})
=&\Bigg[2\pmb{\mathtt{C}}(u)\left(\cosh{u(r_{\eps})}-1\right)-2(N-1)\eps^2\int_{\frac{R}{2}}^{r_{\eps}}\frac{1}{r}u'^2(r)\drr+{e}^{-\frac{R}{16\eps}(\cosh{a}_0)^{-1/2}}\mathcal{O}(1)\Bigg]^{1/2}\notag\\
=&\Bigg[4\left(1-\eps\frac{4N}{R}\sinh^2\frac{b}{4}\right)\sinh^2\frac{u(r_{\eps})}{2}-\eps\frac{16(N-1)}{R}\sinh^2\frac{u(r_{\eps})}{4}+\eps^{3/2}\mathcal{O}(1)\Bigg]^{1/2}\\
=&2\sinh\frac{u(r_{\eps})}{2}\left[1-\frac{\eps}{R}\left(2N\sinh^2\frac{b}{4}+\frac{N-1}{2}\sech^2\frac{u(r_{\eps})}{4}+\sqrt{\eps}\mathcal{O}(1)\right)\right].\notag
\end{align}
We stress here that the quantity $\mathcal{O}(1)$ is uniformly bounded to $r_{\eps}$ 
because $|\pi_{\eps;g}(r_{\eps})|\lesssim\sqrt{\eps}$ for $g(r)=\frac{1}{r}$.
Therefore, (\ref{1119-9011-id}) follows from (\ref{max-u}) and (\ref{chchchha}).

For the convenience, in (\ref{chchchha}) we replace $r_{\eps}$ with $r$ and obtain
\begin{align}\label{chchchha-2}
\frac{u'(r)}{2\sinh\frac{u(r)}{2}}=\frac{1}{\eps}-\frac{2N}{R}\sinh^2\frac{b}{4}-\frac{N-1}{2R}\sech^2\frac{u(r)}{4}
+\sqrt{\eps}\mathcal{O}(1).
\end{align}
Integrating (\ref{chchchha-2}) over $[r_{\eps},R]$ immediately gives
\begin{align}\label{chchchha-3}
\frac{1}{2}\int_{r_{\eps}}^R\frac{{u}'(r)}{\sinh\frac{u(r)}{2}}\drr=&\left(\frac{1}{\eps}-\frac{2N}{R}\sinh^2\frac{b}{4}+\sqrt{\eps}\mathcal{O}(1)\right)(R-r_{\eps})-\frac{N-1}{2R}\int_{r_{\eps}}^R\sech^2\frac{u(r)}{4}\drr.
\end{align}
Note also that $u$ is uniformly bounded to $\eps$ and $\frac{\eps{u}'(r)}{2\sinh\frac{u(r)}{2}}=1+\eps\mathcal{O}(1)$ (by (\ref{chchchha-2})). Hence, one may check that
\begin{align}\label{chchchha-6}
\int_{r_{\eps}}^R\sech^2\frac{u(r)}{4}\drr=&\int_{r_{\eps}}^R\sech^2\frac{u(r)}{4}\left(\frac{\eps{u}'(r)}{2\sinh\frac{u(r)}{2}}+\eps\mathcal{O}(1)\right)\drr\notag\\
=&\log\left|\frac{\tanh\frac{u(R)}{4}}{\tanh\frac{u(r_{\eps})}{4}}\right|+\frac{1}{2}\left(\tanh^2\frac{u(r_{\eps})}{4}-\tanh^2\frac{u(R)}{4}\right)+\eps\mathcal{O}(1)(R-r_{\eps}).
\end{align}
Combining (\ref{chchchha-5}) with (\ref{chchchha-3})--(\ref{chchchha-6}) and passing a calculation directly, one finds
\begin{align*}
&\left(1-\eps\left(\frac{2N}{R}\sinh^2\frac{b}{4}+\sqrt{\eps}\mathcal{O}(1)\right)\right)\frac{R-r_{\eps}}{\eps}\\
=&\left(1+\frac{N-1}{2R}\right)\log\left|\frac{\tanh\frac{u(R)}{4}}{\tanh\frac{u(r_{\eps})}{4}}\right|+\frac{N-1}{4R}\left(\tanh^2\frac{u(r_{\eps})}{4}-\tanh^2\frac{u(R)}{4}\right).
\end{align*}
After making appropriate manipulations, we obtain
\begin{equation}\label{ch0812-2020}
    \begin{aligned}
\frac{R-r_{\eps}}{\eps^2}=&\left(\frac{1}{\eps}+\frac{2N}{R}\sinh^2\frac{b}{4}\right)\\
&\times\Bigg[\left(1+\frac{N-1}{2R}\right)\log\left|\frac{\tanh\frac{u(R)}{4}}{\tanh\frac{u(r_{\eps})}{4}}\right|+\frac{N-1}{4R}\left(\tanh^2\frac{u(r_{\eps})}{4}-\tanh^2\frac{u(R)}{4}\right)\Bigg]+\sqrt{\eps}\mathcal{O}(1).
\end{aligned}
\end{equation}
Here we have used $\displaystyle\liminf_{\eps\downarrow0}u(r_{\eps})>0$ again to get the refined asymptotics of 
$(R-r_{\eps})/\eps$.
Therefore, we prove (\ref{new-r-eps0305})
and completes the proof of Lemma~\ref{lem-1119-9011}.
\end{proof}
  
Note that by Proposition~\ref{prop-ceps}
and Lemma~\ref{asy-radial-lem}, there exists ${r}_{\eps}^{k;\widetilde{k}}\stackrel{\mathrm{as}\,\,\eps\downarrow0}{-\!\!\!-\!\!\!-\!\!\!-\!\!\!\longrightarrow}{R}$
with $\displaystyle\limsup_{\eps\downarrow0}\frac{R-{r}_{\eps}^{k;\widetilde{k}}}{\eps}<\infty$ such that
\begin{equation}\label{0310-2018}
{u}({r}_{\eps}^{k;\widetilde{k}})=k+\Big(\frac{\widetilde{k}}{R}+o_{\eps}(1)\Big)\eps\leq{u}(R)\,\,\mathrm{as}\,\,0<\eps\ll1,
\end{equation}
where $k$ and $\widetilde{k}$ are real numbers independent of $\eps$ and satisfy
one of the following conditions:
\begin{align}\label{mainth1-id8}
\begin{cases}
(a)\,\,0<k<b\,\,\mathrm{and}\,\,\widetilde{k}\in\mathbb{R};\\
(b)\,\,k=b\,\,\mathrm{and}\,\,\widetilde{k}\leq\frac{2\gamma\left(N\cosh^2\frac{b}{2}-1\right)\tanh\frac{b}{4}}{\gamma\cosh\frac{b}{2}+1}.
\end{cases}
\end{align}
Here the second order term having the order $\eps$ is a natural consideration
due to the rigorous derivation of $u(R)$ and $u'(R)$ in Proposition~\ref{prop-ceps}.
Since ${r}_{\eps}^{k;\widetilde{k}}$ may depend on $k$ and $\widetilde{k}$,
 we shall establish asymptotics of ${r}_{\eps}^{k;\widetilde{k}}$
such that
\begin{align}\label{mainth1-id9}
\llang{u}(r_{\eps}^{k;\widetilde{k}})\rrang=k+\frac{\widetilde{k}}{R}\eps\quad\mathrm{as}\,\,0<\eps\ll1. 
\end{align}
More precisely, 
for $r_{\eps}^{k;\widetilde{k}}\in\pmb{\mathbb{B}_{\partial}^{\eps}}$ admiting (\ref{mainth1-id9}), 
asymptotics of $\llang(R-r_{\eps}^{k;\widetilde{k}})/\eps\rrang$
and $\llang{u}'(r_{\eps}^{k;\widetilde{k}})\rrang$
are uniquely determined by $k$ and $\widetilde{k}$,
which can be precisely depicted as follows.

\begin{lemma}\label{mainthm-2}
Under the same hypotheses as in Theorem~\ref{new-cor},
if $r_{\eps}^{k;\widetilde{k}}\in\pmb{\mathbb{B}_{\partial}^{\eps}}$ satisfies (\ref{mainth1-id9}) as $0<\eps\ll1$, then
\begin{align}\label{mainth1-id10}
\llang{u}'(r_{\eps}^{k;\widetilde{k}})\rrang=\frac{2}{\eps}\sinh\frac{k}{2}
-\frac{1}{R}&\left(4N\sinh\frac{k}{2}\sinh^2\frac{b}{4}+2(N-1)\tanh\frac{k}{4}-\widetilde{k}\cosh\frac{k}{2}\right).
\end{align}
Moreover, the exact leading order term $\mathcal{O}(1)$ and the second order term $\eps\mathcal{O}(1)$ of $\llang(R-r_{\eps}^{k;\widetilde{k}})/{\eps}\rrang$ is described by 
\begin{align}\label{2018-0210}
\llang\frac{R-r_{\eps}^{k;\widetilde{k}}}{\eps}\rrang=\mathcal{A}^{k}+\frac{\eps}{2R}\left(4N\mathcal{A}^{k}\sinh^2\frac{b}{4}+\mathcal{B}^{k;\widetilde{k}}\right),
\end{align}
where
\begin{align}
\mathcal{A}^{k}=&\,\left(1+\frac{N-1}{2R}\right)\log\frac{\tanh\frac{b}{4}}{\tanh\frac{k}{4}}+\frac{N-1}{4R}\left(\tanh^2\frac{k}{4}-\tanh^2\frac{b}{4}\right),\label{mathcal-A}
\end{align}
and
\begin{align}\label{mathcal-B}
\mathcal{B}^{k;\widetilde{k}}=&\,\frac{\gamma(N\cosh^2\frac{b}{2}-1)\sech^2\frac{b}{4}}{\gamma\cosh\frac{b}{2}+1}\left(1+\frac{N-1}{2R}\sech^2\frac{b}{4}\right)-\frac{\widetilde{k}}{\sinh\frac{k}{2}}
\left(1+\frac{N-1}{2R}\sech^2\frac{k}{4}\right).
\end{align}
\end{lemma}
\begin{proof} 
We calculate asymptotics of ${u}'(r_{\eps}^{k;\widetilde{k}})$
and $(R-r_{\eps}^{k;\widetilde{k}})/\eps$ as follows, respectively.
\begin{itemize}
 \item[\textbf{(b1).}]\,\,Plug (\ref{0310-2018}) into (\ref{1119-9011-id}), 
\end{itemize} 
\begin{align*}
&u'(r_{\eps}^{k;\widetilde{k}})\\
=&\,2\left(\sinh\frac{k+\left(\frac{\widetilde{k}}{R}+o_{\eps}(1)\right)\eps}{2}\right)\left(\frac{1}{\eps}-\frac{2N}{R}\sinh^2\frac{b}{4}-\frac{N-1}{2R}\sech^{2}\frac{k+\left(\frac{\widetilde{k}}{R}+o_{\eps}(1)\right)\eps}{4}+o_{\eps}(1)\right)\\
=&\left(2\sinh\frac{k}{2}+\frac{\widetilde{k}}{R}\eps\left(\cosh\frac{k}{2}+o_{\eps}(1)\right)\right)\left[\frac{1}{\eps}-\frac{2N}{R}\sinh^2\frac{b}{4}-\frac{N-1}{2R}\left(1-\frac{\widetilde{k}}{2R}\eps\tanh\frac{k}{4}\right)\sech^{2}\frac{k}{4}+o_{\eps}(1)\right]\\
=&\,\frac{2}{\eps}\sinh\frac{k}{2}\!-\!\frac{1}{R}\left(4N\sinh\frac{k}{2}\sinh^2\frac{b}{4}+2(N-1)\tanh\frac{k}{4}-\widetilde{k}\cosh\frac{k}{2}\right)\!+\!o_{\eps}(1).
\end{align*}

Here we have used similar approximations as (\ref{0308-a1}) and (\ref{0308-a3})
to deal with the asymptotics of $u'(r_{\eps}^{k;\widetilde{k}})$.
\begin{itemize}
 \item[\textbf{(b2).}]\,\,Putting (\ref{0310-2018}) into (\ref{new-r-eps0305})
and using asymptotics of $u(R)$ (see (\ref{mainth1-id6})), one can check that
\end{itemize}
{
\begin{align}\label{crucial0308}
&\frac{R-r_{\eps}}{\eps}\notag\\
=&\left(1+\eps\cdot\frac{2N}{R}\sinh^2\frac{b}{4}\right)\Bigg\{\left(1+\frac{N-1}{2R}\right)\log\frac{\tanh\left(\frac{b}{4}+\frac{\eps\gamma\left(N\cosh^2\frac{b}{2}-1\right)\tanh\frac{b}{4}}{2R(\gamma\cosh\frac{b}{2}+1)}+\eps{o}_{\eps}(1)\right)}{\tanh\left(\frac{k}{4}+\frac{\eps\widetilde{k}}{4R}+\eps{o}_{\eps}(1)\right)}\notag\\
&\hspace{6pt}+\frac{N-1}{4R}\left[\tanh^2\left(\frac{k}{4}+\frac{\eps\widetilde{k}}{4R}+\eps{o}_{\eps}(1)\right)-\tanh^2\left(\frac{b}{4}+\frac{\eps\gamma\left(N\cosh^2\frac{b}{2}-1\right)\tanh\frac{b}{4}}{2R(\gamma\cosh\frac{b}{2}+1)}+\eps{o}_{\eps}(1)\right)\right]\Bigg\}+\eps^{3/2}\mathcal{O}(1)\notag\\
=&\left(1+\eps\cdot\frac{2N}{R}\sinh^2\frac{b}{4}\right)\Bigg\{\underbrace{\left(1+\frac{N-1}{2R}\right)
\log\frac{\tanh\frac{b}{4}}{\tanh\frac{k}{4}}+\frac{N-1}{4R}\left(\tanh^2\frac{k}{4}-\tanh^2\frac{b}{4}\right)}_{:=\mathcal{A}^k\,\,(\mathrm{see\,\,(\ref{mathcal-A})})}\notag\\
&\hspace{20pt}+\frac{\eps}{2R}\Bigg[\left(1+\frac{N-1}{2R}\right)\Bigg(\underbrace{\frac{2\gamma(N\cosh^2\frac{b}{2}-1)\tanh\frac{b}{4}}{(\gamma\cosh\frac{b}{2}+1)\sinh\frac{b}{2}}}_{:=I_1}-\underbrace{\frac{\widetilde{k}}{\sinh\frac{k}{2}}}_{:=J_1}\Bigg)\\
&\hspace{50pt}+\frac{N-1}{4R}\Bigg(\underbrace{\frac{\widetilde{k}\tanh\frac{k}{4}}{\cosh^2\frac{k}{4}}}_{:=J_2}-\underbrace{\frac{2\gamma(N\cosh^2\frac{b}{2}-1)\tanh^2\frac{b}{4}}{(\gamma\cosh\frac{b}{2}+1)\cosh^2\frac{b}{4}}}_{:=I_2}\Bigg)\Bigg]\Bigg\}+\eps{o}_{\eps}(1)\notag\\
=&\mathcal{A}^k+\frac{\eps}{2R}\Bigg\{4N\mathcal{A}^k\sinh^2\frac{b}{4}+\underbrace{\frac{\gamma(N\cosh^2\frac{b}{2}-1)\sech^2\frac{b}{4}}{\gamma\cosh\frac{b}{2}+1}\left[\left(1+\frac{N-1}{2R}\right)-\frac{N-1}{2R}\tanh^2\frac{b}{4}\right]}_{:=I_3=\left(1+\frac{N-1}{2R}\right)I_1-\frac{N-1}{4R}I_2}\notag\\
&\hspace{50pt}+\underbrace{\widetilde{k}\left[-\left(1+\frac{N-1}{2R}\right)\frac{1}{\sinh\frac{k}{2}}+\frac{N-1}{4R}\tanh\frac{k}{4}\sech^2\frac{k}{4}\right]}_{:=J_3=-\left(1+\frac{N-1}{2R}\right)J_1+\frac{N-1}{4R}J_2}\Bigg\}+\eps{o}_{\eps}(1)\notag\\
=&\mathcal{A}^k+\frac{\eps}{2R}\Bigg(4N\mathcal{A}^k\sinh^2\frac{b}{4}+\mathcal{B}^{k;\widetilde{k}}\Bigg)+\eps{o}_{\eps}(1),\notag
\end{align}}
where (cf. (\ref{mathcal-B}))
\begin{align*}
\mathcal{B}^{k;\widetilde{k}}:=&\,I_3+J_3=\frac{\gamma(N\cosh^2\frac{b}{2}-1)\sech^2\frac{b}{4}}{\gamma\cosh\frac{b}{2}+1}\left(1+\frac{N-1}{2R}\sech^2\frac{b}{4}\right)-\frac{\widetilde{k}}{\sinh\frac{k}{2}}
\left(1+\frac{N-1}{2R}\sech^2\frac{k}{4}\right),
\end{align*}
which is obtained from the identity 
\begin{align*}
J_3=-\left(1+\frac{N-1}{2R}\right)\frac{1}{\sinh\frac{k}{2}}+\frac{N-1}{4R}\tanh\frac{k}{4}\sech^2\frac{k}{4}=-\frac{1}{\sinh\frac{k}{2}}\left(1+\frac{N-1}{2R}\sech^2\frac{k}{4}\right).
\end{align*}
For the second equality of (\ref{crucial0308}), we have applied
some elementary approximations
\begin{equation*}
 \tanh(\alpha+(\beta+o_{\eps}(1))\eps)\approx\tanh\alpha+\big(\beta\sech^2\alpha+o_{\eps}(1)\big)\eps
\end{equation*}
and 
\begin{equation*}
\log\tanh(\alpha+(\beta+o_{\eps}(1))\eps)\approx\log\tanh\alpha+\Big(\frac{2\beta}{\sinh2\alpha}+o_{\eps}(1)\Big)\eps. 
\end{equation*}
Therefore, we prove (\ref{mainth1-id10}) and (\ref{2018-0210}) and complete the proof of Lemma~\ref{mainthm-2}.
\end{proof}

Now we are in a position to prove Theorem~\ref{new-cor}.
\begin{proof}[Proof of Theorem~\ref{new-cor}] 
Let
\begin{align}\label{newlet0317}
\overline{r^{\eps}_{\mathtt{p};\mathtt{q}}}=R\left(1-\frac{\mathtt{p}}{R}\eps-\frac{\mathtt{q}}{R^2}\eps^2\right)\in\pmb{\mathbb{B}_{\partial}^{\eps}}.
\end{align}
Then for any $r^{\eps}_{\mathtt{p};\mathtt{q}}\in\pmb{\mathbb{B}_{\partial}^{\eps}}$, we have 
\begin{equation}\label{0328-2018-haha}
\lim_{\eps\downarrow0}\eps^{-2}(r^{\eps}_{\mathtt{p};\mathtt{q}}-\overline{r^{\eps}_{\mathtt{p};\mathtt{q}}})=0.
\end{equation}
 Since $u$ and $\eps{u'}$
are uniformly bounded to $\eps$ (by (\ref{gradesti})), together with (\ref{0328-2018-haha})
one immediately finds $\llang{u}(r^{\eps}_{\mathtt{p};\mathtt{q}})\rrang=\llang{u}(\overline{r^{\eps}_{\mathtt{p};\mathtt{q}}})\rrang$ and $\llang{u}'(r^{\eps}_{\mathtt{p};\mathtt{q}})\rrang=\llang{u}'(\overline{r^{\eps}_{\mathtt{p};\mathtt{q}}})\rrang$
as $0<\eps\ll1$.
 Hence, to prove Theorem~\ref{new-cor},
it suffices to establish asymptotics of $\llang{u}(\overline{r^{\eps}_{\mathtt{p};\mathtt{q}}})\rrang$ and $\llang{u}'(\overline{r^{\eps}_{\mathtt{p};\mathtt{q}}})\rrang$.

Regarding $\mathcal{A}^{k}$ (defined in (\ref{mathcal-A})) as a function of $k$ in $(0,b]$, one may check 
 that $\displaystyle\lim_{k\downarrow0}\mathcal{A}^{k}=\infty$, $\mathcal{A}^{b}=0$
and $\mathcal{A}^{k}$  
 is strictly decreasing to $k$ in $(0,b)$. 
Note also that $\mathcal{B}^{k;\widetilde{k}}$ (defined in (\ref{mathcal-B})) is a linear function of $\widetilde{k}$.
Hence by (\ref{2018-0210})--(\ref{mathcal-B}), we obtain that for any $\mathtt{p}\geq0$ and $\mathtt{q}\in\mathbb{R}$
 there uniquely exist $k=k(\mathtt{p})$ and $\widetilde{k}=\widetilde{k}(\mathtt{p},\mathtt{q})$
satisfying (\ref{mainth1-id8}) such that (\ref{0308-afr}) and (\ref{mainth1-id9}) hold, i.e.,
\begin{equation}\label{kandwk}
 \mathcal{A}^{k(\mathtt{p})}=\mathtt{p}\quad\mathrm{and}\quad\frac{1}{2}\left(4N\mathtt{p}\sinh^2\frac{b}{4}+\mathcal{B}^{k(\mathtt{p});\widetilde{k}(\mathtt{p},\mathtt{q})}\right)=\mathtt{q}
\end{equation}
and
\begin{align}\label{u-kandwk}
\llang{u}(r_{\eps}^{k(\mathtt{p});\widetilde{k}(\mathtt{p},\mathtt{q})})\rrang=k(\mathtt{p})+\frac{\eps}{R}\widetilde{k}(\mathtt{p},\mathtt{q})\quad\mathrm{as}\,\,0<\eps\ll1. 
\end{align} 
From the second equation of (\ref{kandwk}), one gets
\begin{align}\label{0308-wk}
\widetilde{k}(\mathtt{p},\mathtt{q})=\mathcal{H}^{\gamma;b}_{\mathtt{p};\mathtt{q}}\sinh\frac{k(\mathtt{p})}{2},
\end{align}
where $\mathcal{H}^{\gamma;b}_{\mathtt{p};\mathtt{q}}$ is defined in (\ref{mathcal-H}).
Moreover, by (\ref{newlet0317})--(\ref{kandwk}) and (\ref{u-kandwk})--(\ref{0308-wk}) we have 
\begin{align*}
\llang\frac{R-{r}_{\eps}^{k(\mathtt{p});\widetilde{k}(\mathtt{p},\mathtt{q})}}{\eps}\rrang=\llang\frac{R-\overline{r^{\eps}_{\mathtt{p};\mathtt{q}}}}{\eps}\rrang
\end{align*}
and
\begin{align*}
\llang{u}(r_{\eps}^{k(\mathtt{p});\widetilde{k}(\mathtt{p},\mathtt{q})})\rrang=k(\mathtt{p})+\frac{\eps}{R}\mathcal{H}^{\gamma;b}_{\mathtt{p};\mathtt{q}}\sinh\frac{k(\mathtt{p})}{2}.
\end{align*}
As a consequence, $\overline{r^{\eps}_{\mathtt{p};\mathtt{q}}}-r_{\eps}^{k(\mathtt{p});\widetilde{k}(\mathtt{p},\mathtt{q})}=\eps^2{o}_{\eps}(1)$ and 
\begin{align}\label{0317-910}
{u}(\overline{r^{\eps}_{\mathtt{p};\mathtt{q}}})=&{u}(r_{\eps}^{k(\mathtt{p});\widetilde{k}(\mathtt{p},\mathtt{q})})+u'(\theta^{\eps})(\overline{r^{\eps}_{\mathtt{p};\mathtt{q}}}-r_{\eps}^{k(\mathtt{p});\widetilde{k}(\mathtt{p},\mathtt{q})})\notag\\[-0.7em]
&\\[-0.7em]
=&k(\mathtt{p})+\frac{\eps}{R}\left(\mathcal{H}^{\gamma;b}_{\mathtt{p};\mathtt{q}}\sinh\frac{k(\mathtt{p})}{2}+{o}_{\eps}(1)\right),\notag
\end{align}
where $\theta^{\eps}$ lies between $\overline{r^{\eps}_{\mathtt{p};\mathtt{q}}}$ and $r_{\eps}^{k(\mathtt{p});\widetilde{k}(\mathtt{p},\mathtt{q})}$. Here we have used (\ref{gradesti}) to assert 
\begin{equation*}
u'(\theta^{\eps})\left(\overline{r^{\eps}_{\mathtt{p};\mathtt{q}}}-r_{\eps}^{k(\mathtt{p});\widetilde{k}(\mathtt{p},\mathtt{q})}\right)=\eps{o}_{\eps}(1).
\end{equation*}
Hence, (\ref{u-0d}) follows from (\ref{0317-910}) and $\llang{u}(\overline{r^{\eps}_{\mathtt{p};\mathtt{q}}})\rrang=\llang{u}(r^{\eps}_{\mathtt{p};\mathtt{q}})\rrang$.

Comparing (\ref{mainth1-id9}) to the first two order terms of (\ref{u-0d}),
we shall put $k=k(\mathtt{p})$ and $\widetilde{k}=\frac{\mathcal{H}^{\gamma;b}_{\mathtt{p};\mathtt{q}}}{R}\sinh\frac{k(\mathtt{p})}{2}$ into (\ref{mainth1-id10}), and therefore obtain
\begin{align}
u'(r^{\eps}_{\mathtt{p};\mathtt{q}})
=\frac{2}{\eps}\sinh\frac{k(\mathtt{p})}{2}-\frac{1}{R}\Bigg(&4N\sinh^2\frac{b}{4}\sinh\frac{k(\mathtt{p})}{2}\notag\\
&+2(N-1)\tanh\frac{k(\mathtt{p})}{4}-\frac{1}{2}\mathcal{H}^{\gamma;b}_{\mathtt{p};\mathtt{q}}\sinh{k(\mathtt{p})}\Bigg)+o_{\eps}(1).\notag
\end{align}
Therefore, we get (\ref{u-1std}).

Finally, we need to check their uniform convergence when $\mathtt{p}$ is located in a bounded interval $[0,\mathtt{p}^*]$. By
(\ref{0308-afr}), for any $\mathtt{q}\in\mathbb{R}$ 
 one finds $\ds\lim_{\eps\downarrow0}\sup_{\mathtt{p}\in[0,\mathtt{p}^*]}\left(R-r^{\eps}_{\mathtt{p};\mathtt{q}}\right)/\eps\leq\mathtt{p}^*$. Along with Theorem~\ref{bdylayer-exist-thm},
 we have
\begin{align*}
 \lim_{\eps\downarrow0}\inf_{\mathtt{p}\in[0,\mathtt{p}^*]}\sinh{u}(r^{\eps}_{\mathtt{p};\mathtt{q}})
\geq\liminf_{\eps\downarrow0}\sinh{u}(r^{\eps}_{\mathtt{p}^*;\mathtt{q}})>0. 
\end{align*}
As a consequence, $\sup_{\mathtt{p}\in[0,\mathtt{p}^*]}\frac{1}{\sinh{u}(r^{\eps}_{\mathtt{p};\mathtt{q}})}$ is uniformly bounded to $\eps$ as $0<\eps\ll1$, and all arguments involving the pointwise estimates of 
$u(r^{\eps}_{\mathtt{p};\mathtt{q}})$ and $u'(r^{\eps}_{\mathtt{p};\mathtt{q}})$ 
can be improved so that the convergence (\ref{0403-2018}) 
is uniformly in $[0,\mathtt{p}^*]\times\mathbb{R}$.
This completes the proof of Theorem~\ref{new-cor}.
\end{proof}

\subsection{Comparison of nonlocal and standard elliptic sinh--Gordon equations}\label{sec-comparison}
\noindent

 For (\ref{eq2})--(\ref{bd2}), recall the nonlocal coefficient $\pmb{\mathtt{C}}(u)\sim1$ as $0<\eps\ll1$.
Hence, as $\eps\downarrow0$, $u$ formally approaches the solution $v$ of the standard elliptic sinh--Gordon equation
\begin{align}
\eps^2\left(v''(r)+\frac{N-1}{r}v'(r)\right)\,=\,\sinh{v},&\quad{r}\in(0,R),\label{v-eqn}\\
v'(0)=0,\quad{v}(R)+{\gamma\eps}v'(R)={a}_0,&\label{v-bdy}
\end{align}
where the condition of $\gamma$ and $a_0$ are same as that in (\ref{bd2}).
For the sake of completeness, we shall compare the pointwise asymptotics of $u$ and $v$ in
the whole domain~$[0,R]$.

 Following the same argument of Lemma~\ref{asy-radial-lem}, it is easy to obtain
\begin{align}
|(u-v)(r)|+\eps|(u-v)'(r)|\lesssim{e}^{-\frac{1}{8\eps}(\cosh{a}_0)^{-1/2}(R-r)},\quad{r}\in[0,R].\notag
\end{align}
However, as $0<\eps\ll1$,
 $u$ and $v$ have different asymptotic behavior near the boundary
since $\pmb{\mathtt{C}}(u)$ is not identically equal to $1$.
Alternatively, note that making the following replacements in (\ref{eq2})--(\ref{bd2}): 
\begin{align}
&\boxed{\displaystyle\eps\mapsto\,{\eps}{\sqrt{\pmb{\mathtt{C}}(u)}}\,\,=\!=\!=\!=\!=\displaystyle\eps-\frac{N}{R}\left(\cosh\frac{b}{2}-1\right)\eps^2+{o}_{\eps}(1)\eps^2}\notag\\[-1.55em]
&\quad\quad\quad\quad\quad\quad\quad\,\,\,\Uparrow\notag\\[-0.8em]
&\quad\quad\quad\quad\quad\quad\,\,\,\,{}_{\mathrm{by\,\,(\ref{ceps-1})}}\quad\quad\quad\quad\quad\quad\quad\quad\quad\quad\quad\quad\quad\quad\quad\,\,,\notag\\[-0.3em]
&\quad\quad\quad\quad\quad\quad\quad\,\,\,\Downarrow\notag\\[-1.55em]
&\boxed{\displaystyle\gamma\mapsto\,\frac{\gamma}{\sqrt{\pmb{\mathtt{C}}(u)}}\,\,=\!=\!=\!=\!=\,\gamma+\displaystyle\frac{N}{R}\gamma\left(\cosh\frac{b}{2}-1\right)\eps+{o}_{\eps}(1)\eps}\notag
\end{align}
one can transform (\ref{eq2}) and (\ref{bd2}) into (\ref{v-eqn}) and (\ref{v-bdy}), respectively.
 Accordingly, it is expected that asymptotic expansions of $u(R)$ and $v(R)$ with respect to $\eps$
have different second order terms. To see the difference,
we can use the same arguments in Sections~\ref{sec-dtn} and \ref{sec-pf-0315-2018}--\ref{sec-mainthm-2}
to get the asymptotics of $v$ and $v'$ in $[0,R]-\pmb{\mathbb{B}_{\partial}^{\eps}}$
and $\pmb{\mathbb{B}_{\partial}^{\eps}}$, respectively. 
The following lemmas can be proved following the same arguments in Lemmas~\ref{lem-1119-9011} and \ref{mainthm-2}
so we omit the proof. The reader can compare these results (of $v$) with
Lemmas~\ref{lem-1119-9011} and \ref{mainthm-2} (of $u$) directly. 

\begin{figure}[htp]
\centering{%
\begin{tabular}{@{\hspace{-0pc}}c@{\hspace{-0pc}}c}
 \psfig{figure=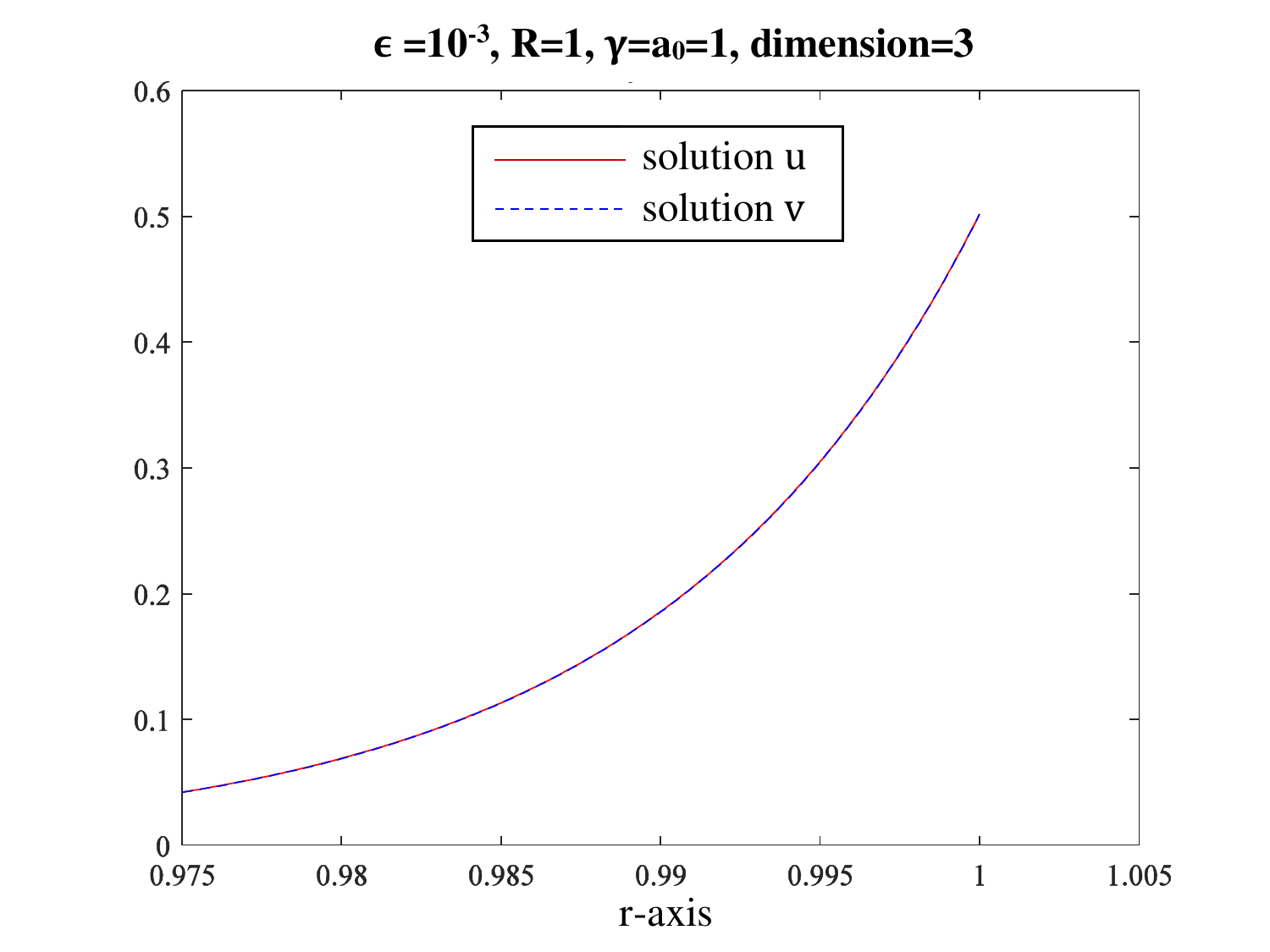, width=10cm}
  \end{tabular}
	}
	\caption{\em Numerical solutions $u$ of (\ref{eq2})--(\ref{bd2}) and $v$ of (\ref{v-eqn})--(\ref{v-bdy})
	with $\eps=10^{-3}$ near the boundary $R=1$.\label{chlee04072018}}
\end{figure}

\begin{lemma}\label{v-lem1}
For $\eps>0$, let $v$ be the unique classical solution of (\ref{v-eqn})--(\ref{v-bdy}), where $a_0$ and $\gamma$ are positive constants independent of $\eps$. Then $v$ and $v'$ are strictly positive in $(0,R]$,
and $v''(r)\geq0$ for $r\in(0,R]$. In addition,
for $r_{\eps}\in\pmb{\mathbb{B}_{\partial}^{\eps}}$, as $0<\eps\ll1$ there hold 
\begin{align}
\left|v'(r_{\eps})-2\sinh\frac{v(r_{\eps})}{2}\left(\frac{1}{\eps}-\frac{N-1}{2R}\sech^2\frac{v(r_{\eps})}{4}\right)\right|\lesssim\sqrt{\eps}\notag
\end{align}
and
{\small
\begin{align}
&\Bigg|\frac{R-r_{\eps}}{\eps}-
\Bigg[\left(1+\frac{N-1}{2R}\right)\log\left|\frac{\tanh\frac{v(R)}{4}}{\tanh\frac{v(r_{\eps})}{4}}\right|+\frac{N-1}{4R}\left(\tanh^2\frac{v(r_{\eps})}{4}-\tanh^2\frac{v(R)}{4}\right)\Bigg]\Bigg|
\lesssim\eps^{3/2}.\notag
\end{align}}
\end{lemma}

\begin{lemma}\label{v-lem2}
Under the same hypotheses as in Lemma~\ref{v-lem1},
if $r_{\eps}^{k;\widetilde{k}}\in\pmb{\mathbb{B}_{\partial}^{\eps}}$ satisfies 
\begin{equation}
\langg{v}(r_{\eps}^{k;\widetilde{k}})\rangg=k+\frac{\widetilde{k}}{R}\eps\,\,as\,\,0<\eps\ll1,\notag
\end{equation}
 then
\begin{align}\notag
\llang{v}'(r_{\eps}^{k;\widetilde{k}})\rrang=
\frac{2}{\eps}\sinh\frac{k}{2}-\frac{1}{R}\left(2(N-1)\tanh\frac{k}{4}-\widetilde{k}\cosh\frac{k}{2}\right).
\end{align}
Moreover, the exact leading order term $\mathcal{O}(1)$ and the second order term $\eps\mathcal{O}(1)$ of $\llang(R-r_{\eps}^{k;\widetilde{k}})/{\eps}\rrang$ is described by 
\begin{align}\notag
\llang\frac{R-r_{\eps}^{k;\widetilde{k}}}{\eps}\rrang=\mathcal{A}^{k}+\frac{\mathcal{B}_{\#}^{k;\widetilde{k}}}{2R}\eps,
\end{align}
where $\mathcal{A}^{k}$ is defined in (\ref{mathcal-A}) and 
\begin{align}\notag
\mathcal{B}_{\#}^{k;\widetilde{k}}=&\,\frac{\gamma(N-1)\sech^2\frac{b}{4}}{\gamma\cosh\frac{b}{2}+1}\left(1+\frac{N-1}{2R}\sech^2\frac{b}{4}\right)-\frac{\widetilde{k}}{\sinh\frac{k}{2}}
\left(1+\frac{N-1}{2R}\sech^2\frac{k}{4}\right).
\end{align}
\end{lemma}

Using Lemmas~\ref{v-lem1}--\ref{v-lem2} and
following the similar arguments in the proof of Theorem~\ref{new-cor}, we establish
refined structure of the boundary layer of $v$ in $[0,R]$ as follows.
\begin{theorem}\label{v-finthm}
Under the same hypotheses as in Lemma~\ref{v-lem1},
as $0<\eps\ll1$, we have
\begin{align}\label{v-329-night1}
 \max\left\{|v(r)|,\,{\gamma\eps}|v'(r)|\right\}\leq2a_0{e}^{-\frac{1}{8\eps}({\cosh{a}_0})^{-1/2}(R-r)},\,\,\forall\,r\in[0,R],
\end{align}
and 
\begin{align}
\langg{v}(R)\rangg=&\,b+\frac{N-1}{R}\eps\cdot\frac{2\gamma\tanh\frac{b}{4}}{\gamma\cosh\frac{b}{2}+1},\label{v-329-night2}\\
\langg{v}'(R)\rangg=&\frac{2}{\eps}\sinh\frac{b}{2}-\frac{N-1}{R}\cdot\frac{2\tanh\frac{b}{4}}{\gamma\cosh\frac{b}{2}+1}.\label{v-329-night3}
\end{align}
Moreover, for $r^{\eps}_{\mathtt{p};\mathtt{q}}\in\pmb{\mathbb{B}_{\partial}^{\eps}}$ obeying (\ref{0308-afr}),
we have
\begin{align}
\langg{v}(r^{\eps}_{\mathtt{p};\mathtt{q}})\rangg=&\,k(\mathtt{p})+\frac{\eps}{R}\mathcal{H}^{\gamma;b}_{\mathtt{p},\mathtt{q},\#}\sinh\frac{k(\mathtt{p})}{2}
\end{align}
and
\begin{align}
\langg{v}'(r^{\eps}_{\mathtt{p};\mathtt{q}})\rangg=&\,2\sinh\frac{k(\mathtt{p})}{2}\left[\frac{1}{\eps}-\frac{1}{R}\left(\frac{N-1}{2}\sech^2\frac{k(\mathtt{p})}{4}-\frac{\mathcal{H}^{\gamma;b}_{\mathtt{p},\mathtt{q},\#}}{2}\cosh\frac{k(\mathtt{p})}{2}\right)\right],
\end{align}
where $k(\mathtt{p})$ is uniquely determined by (\ref{0308-k}) and
\begin{align}
\mathcal{H}^{\gamma;b}_{\mathtt{p},\mathtt{q},\#}=\frac{\gamma(N-1)\sech^2\frac{b}{4}}{\gamma\cosh\frac{b}{2}+1}\cdot\frac{1+\frac{N-1}{2R}\sech^2\frac{b}{4}}{1+\frac{N-1}{2R}\sech^2\frac{k(\mathtt{p})}{4}}-\frac{2\mathtt{q}}{1+\frac{N-1}{2R}\sech^2\frac{k(\mathtt{p})}{4}}.
\end{align}
\end{theorem}

 Using Theorem~\ref{new-cor} and Theorem~\ref{v-finthm}, we may
 compare the difference between asymptotics of $u$ and $v$ in $\pmb{\mathbb{B}_{\partial}^{\eps}}$
as follows:
\begin{corollary}[Comparison of asymptotics of $u$ and $v$]\label{cor-0403uv}
For $\eps>0$, let $u$ be the unique classical solution of
the nonlocal model~(\ref{eq2})--(\ref{bd2})
and $v$ be the model of (\ref{v-eqn})--(\ref{v-bdy}),
where $a_0$ and $\gamma$ are positive constants independent of $\eps$. Then
\begin{align}
\lim_{\eps\downarrow0}\frac{u(R)-v(R)}{\eps}=&\,-\gamma\lim_{\eps\downarrow0}\left(u'(R)-v'(R)\right)
=\frac{N}{R}\cdot\frac{2\gamma\sinh\frac{b}{2}}{\gamma\cosh\frac{b}{2}+1}\left(\cosh\frac{b}{2}-1\right),\notag\\
\lim_{\eps\downarrow0}\frac{u(r^{\eps}_{\mathtt{p};\mathtt{q}})-v(r^{\eps}_{\mathtt{p};\mathtt{q}})}{\eps}=&\,\frac{\mathcal{H}^{\gamma;b}_{\mathtt{p},\mathtt{q}}-\mathcal{H}^{\gamma;b}_{\mathtt{p},\mathtt{q},\#}}{R}\sinh\frac{k(\mathtt{p})}{2}\notag\\
=&\,\frac{4N\sinh^2\frac{b}{4}\sinh\frac{k(\mathtt{p})}{2}}{R+\frac{N-1}{2}\sech^2\frac{k(\mathtt{p})}{4}}\left(\frac{\gamma\left(1+\frac{N-1}{2R}\sech^2\frac{b}{4}\right)}{\gamma\cosh\frac{b}{2}+1}+\mathtt{p}\right),\notag
\end{align}
and
\begin{equation}\label{0406-2018-1}
    \begin{aligned}
&\lim_{\eps\downarrow0}\left(u'(r^{\eps}_{\mathtt{p};\mathtt{q}})-v'(r^{\eps}_{\mathtt{p};\mathtt{q}})\right)\\
=&\frac{2}{R}\sinh\frac{k(\mathtt{p})}{2}\left(-2N\sinh^2\frac{b}{4}+\frac{\mathcal{H}^{\gamma;b}_{\mathtt{p},\mathtt{q}}-\mathcal{H}^{\gamma;b}_{\mathtt{p},\mathtt{q},\#}}{2}\cosh\frac{k(\mathtt{p})}{2}\right)\\
=&\,-\frac{4N}{R}\sinh^2\frac{b}{4}\sinh\frac{k(\mathtt{p})}{2}\Bigg[1-\frac{\cosh\frac{k(\mathtt{p})}{2}}{1+\frac{N-1}{2R}\sech^2\frac{k(\mathtt{p})}{4}}\left(\frac{\gamma\left(1+\frac{N-1}{2R}\sech^2\frac{b}{4}\right)}{\gamma\cosh\frac{b}{2}+1}+\mathtt{p}\right)\Bigg].
\end{aligned}
\end{equation}

\end{corollary}

\begin{remark}\label{0403-rk-2018}
The numerical simulation (see Figure~\ref{chlee04072018}) shows that solutions of $u$ and $v$ with $\eps=10^{-3}$
are almost overlapping near the boundary.
However, we want to stress that
 as $\eps\downarrow0$, even if both $u(r_{\eps})$ and $v(r_{\eps})$ have the same leading order term for 
$r_\eps\in\pmb{\mathbb{B}_{\partial}^{\eps}}$, their slopes always have an $\mathcal{O}(1)$ difference (see (\ref{0406-2018-1})).
\end{remark}

\section{Boundary concentration phenomenon: Proof of Theorem~\ref{mainthm-1}}\label{sec-concentration}
\noindent

 In this section, we give the proof of Theorem~\ref{mainthm-1} as follows.

  Let $h\in\mathrm{C}([0,R])$ be a continuous function independent of $\eps$.
To prove (\ref{mainth1-id1-20180204}), we need to estimate
\begin{align}\label{f1-0328-2018}
\int_0^R\frac{\mathcal{F}(\eps{u}'(r))-\mathcal{F}(0)}{\eps}h(r)\drr,\quad\mathrm{as}\,\,0<\eps\ll1.
\end{align}
Note that (\ref{0217-0217}) implies 
\begin{align}\label{r-0328-f}
\frac{\left|\mathcal{F}(\eps{u}'(r))-\mathcal{F}(0)\right|}{\eps}\lesssim\eps^{-1}{e}^{-\frac{M_1}{\eps}\tau(R-r)},\quad{r}\in[0,R].
\end{align}
Thus, the main difficulty is to deal with the estimate $\frac{\mathcal{F}(\eps{u}'(r))-\mathcal{F}(0)}{\eps}$
as $r$ is quite close to the boundary. Note also that $0<\tau\leq1$. Due to (\ref{r-0328-f}),
we shall consider a decomposition of (\ref{f1-0328-2018}) as follows:
\begin{align}\label{f1-0328-2018-ad1}
\int_0^R\frac{\mathcal{F}(\eps{u}'(r))-\mathcal{F}(0)}{\eps}h(r)\drr
=&\int_0^{R-{\eps}^{1-\tau/2}}\frac{\mathcal{F}(\eps{u}'(r))-\mathcal{F}(0)}{\eps}h(r)\drr\notag\\
&+\int_{R-{\eps}^{1-\tau/2}}^R\frac{\mathcal{F}(\eps{u}'(r))-\mathcal{F}(0)}{\eps}(h(r)-h(R))\drr\\
&+h(R)\int_{R-{\eps}^{1-\tau/2}}^R\frac{\mathcal{F}(\eps{u}'(r))-\mathcal{F}(0)}{\eps}\drr.\notag
\end{align}
Then by (\ref{r-0328-f}) one may check that
\begin{align*}
&\left|\int_0^{R-{\eps}^{1-\tau/2}}\frac{\mathcal{F}(\eps{u}'(r))-\mathcal{F}(0)}{\eps}h(r)\drr\right|
\lesssim
\eps^{-1}\left(\max_{[0,R]}|h|\right)\int_0^{R-{\eps}^{1-\tau/2}}{e}^{-\frac{M_1}{\eps}\tau(R-r)}\drr\lesssim{e}^{-\frac{M_1\tau}{\eps^{\tau/2}}},
\end{align*}
and
\begin{align*}
\left|\int_{R-{\eps}^{1-\tau/2}}^R\frac{\mathcal{F}(\eps{u}'(r))-\mathcal{F}(0)}{\eps}(h(r)-h(R))\drr\right|\lesssim\max_{r\in[R-{\eps}^{1-\tau/2},R]}|h(r)-h(R)|.
\end{align*}
As a consequence,
\begin{align}\label{0329-mornight}
\lim_{\eps\downarrow0}&\left(\int_0^{R-{\eps}^{1-\tau/2}}\frac{\mathcal{F}(\eps{u}'(r))-\mathcal{F}(0)}{\eps}h(r)\drr+\int_{R-{\eps}^{1-\tau/2}}^R\frac{\mathcal{F}(\eps{u}'(r))-\mathcal{F}(0)}{\eps}(h(r)-h(R))\drr\right)=0.
\end{align}

To deal with the rightmost-hand side of (\ref{f1-0328-2018-ad1}),
we first notice that by (\ref{ceps-1}), (\ref{gradesti}) and (\ref{cr-in2-add}),
there holds $\big|\eps{u}'(r)-2\sinh\frac{u(r)}{2}\big|\leq{C}\sqrt{\eps}$ uniformly in $[0,R]$ as $0<\eps\ll1$,
where $C>0$ is independent of $\eps$. For the sake of convenience,
we express it by
\begin{align}\label{new-express}
\eps{u}'(r)=2\sinh\frac{u(r)}{2}+o_{\eps}(1),\quad\mathrm{as}\,\,0<\eps\ll1.
\end{align}
Note also that $u(R)\to{b}$ as $\eps\downarrow0$,
 $u(r)>0$ and $u'(r)>0$ for $r\in(0,R]$ (cf. Lemma~\ref{asy-radial-lem}(i)).
Hence, by (\ref{new-express}),
\begin{equation}\label{0329-lai}
  \begin{aligned}
&\int_{R-{\eps}^{1-\tau/2}}^R\frac{\mathcal{F}(\eps{u}'(r))-\mathcal{F}(0)}{\eps}\drr\\
&\hspace{16pt}=\int_{R-{\eps}^{1-\tau/2}}^R\frac{\mathcal{F}(2\sinh\frac{u(r)}{2}+o_{\eps}(1))-\mathcal{F}(0)}{2\sinh\frac{u(r)}{2}+o_{\eps}(1)}u'(r)\drr\\
&\hspace{16pt}=\int_{u(R-{\eps}^{1-\tau/2})}^{u(R)}\frac{\mathcal{F}(2\sinh\frac{t}{2}+o_{\eps}(1))-\mathcal{F}(0)}{2\sinh\frac{t}{2}+o_{\eps}(1)}\dtt\\
&\hspace{16pt}=\left\{\int^{0+}_{u(R-{\eps}^{1-\tau/2})}+\int_{0+}^b+\int_b^{u(R)}\right\}\frac{\mathcal{F}(2\sinh\frac{t}{2}+o_{\eps}(1))-\mathcal{F}(0)}{2\sinh\frac{t}{2}+o_{\eps}(1)}\dtt.
\end{aligned}  
\end{equation}

We need to estimate the last expression of \eqref{0329-lai}. Obviously, as $\eps\downarrow0$,
\begin{align}\label{329-001}
&\int_{0+}^b\frac{\mathcal{F}(2\sinh\frac{t}{2}+o_{\eps}(1))-\mathcal{F}(0)}{2\sinh\frac{t}{2}+o_{\eps}(1)}\dtt
\to
\int_{0+}^b\frac{\mathcal{F}(2\sinh\frac{t}{2})-\mathcal{F}(0)}{2\sinh\frac{t}{2}}\dtt\,\,\mathrm{(which\,\,is\,\,finite)}.
\end{align}
On the other hand,
\begin{equation}\label{329-002}
   \begin{aligned}
 &\left|\left\{\int^{0+}_{u(R-{\eps}^{1-\tau/2})}+\int_b^{u(R)}\right\}\frac{\mathcal{F}(2\sinh\frac{t}{2}+o_{\eps}(1))-\mathcal{F}(0)}{2\sinh\frac{t}{2}+o_{\eps}(1)}\dtt\right|\\
&\hspace{66pt}\lesssim
\left\{\int_{0+}^{u(R-{\eps}^{1-\tau/2})}+\int_b^{u(R)}\right\}\left(2\sinh\frac{t}{2}\right)^{\tau-1}\dtt+o_{\eps}(1)\\
&\hspace{66pt}\lesssim\left\{\int_{0+}^{u(R-{\eps}^{1-\tau/2})}+\int_b^{u(R)}\right\}t^{\tau-1}\dtt+o_{\eps}(1)\\
&\hspace{66pt}\lesssim\tau^{-1}\left(u^{\tau}(R-{\eps}^{1-\tau/2})+u^{\tau}(R)-b^{\tau}+o_{\eps}(1)\right)\to0\,\,\mathrm{as}\,\,\eps\downarrow0.
\end{aligned}
\end{equation}

Here we have used $0<\tau\leq1$, $\sinh\frac{t}{2}\geq\frac{t}{2}$ for $t\geq0$,
$u(R-{\eps}^{1-\tau/2})\to0$ (by (\ref{gradesti})) and $u(R)\to{b}$ as $\eps\downarrow0$. 
As a result,
\begin{align}\label{over329}
\lim_{\eps\downarrow0}\int_{R-{\eps}^{1-\tau/2}}^R\frac{\mathcal{F}(\eps{u}'(r))-\mathcal{F}(0)}{\eps}\drr=\int_{0+}^b\frac{\mathcal{F}(2\sinh\frac{t}{2})-\mathcal{F}(0)}{2\sinh\frac{t}{2}}\dtt
\end{align}
which follows from (\ref{0329-lai})--(\ref{329-002}).
Combining (\ref{f1-0328-2018-ad1})--(\ref{0329-mornight}) and (\ref{over329}),
we get (\ref{mainth1-id1-20180204}) and complete the proof of Theorem~\ref{mainthm-1}(I-i). Moreover, by applying the same argument, we can prove Theorem~\ref{mainthm-1}(I-ii).

Now we want to prove (\ref{mainth1-id1-20180319}).
By (\ref{0315-2018}), we have $\ds\lim_{\eps\downarrow0}\inf_{[r_{\mathtt{p}}^{\eps},R]}u>0$
and $\ds\lim_{\eps\downarrow0}\inf_{[r_{\mathtt{p}}^{\eps},R]}\eps{u}'>0$. This along with (\ref{chchchha-2})
concludes 
\begin{align}\label{chchchha-2018329}
\lim_{\eps\downarrow0}\sup_{[r_{\mathtt{p}}^{\eps},R]}\left(\frac{{u}'(r)}{2\sinh\frac{u(r)}{2}}-\frac{1}{\eps}\right)<\infty.
\end{align}
Since $\left|R-r_{\mathtt{p}}^{\eps}\right|\approx\mathtt{p}{\eps}$,
by putting (\ref{chchchha-2018329}) into (\ref{mainth1-id1-20180319}) 
and using the same argument of (\ref{f1-0328-2018-ad1})--(\ref{329-002}),
after making appropriate manipulations we obtain
\begin{align}\label{mainth1-id1-20180329}
\int_0^R\frac{\mathcal{F}(\eps{u}'(r))}{\eps}\chi_{[r_{\mathtt{p}}^{\eps},R]}(r)h(r)\drr
=&\,h(R)\int_{r_{\mathtt{p}}^{\eps}}^R\frac{\mathcal{F}(2\sinh\frac{u(r)}{2})}{2\sinh\frac{u(r)}{2}}u'(r)\drr+o_{\eps}(1)\notag\\
=&\,h(R)\int_{u(r_{\mathtt{p}}^{\eps})}^{u(R)}\frac{\mathcal{F}(2\sinh\frac{t}{2})}{2\sinh\frac{t}{2}}\dtt+o_{\eps}(1)\\
=&\,h(R)\int_{k(\mathtt{p})}^{b}\frac{\mathcal{F}(2\sinh\frac{t}{2})}{2\sinh\frac{t}{2}}\dtt+o_{\eps}(1)\,\,\mathrm{as}\,\,0<\eps\ll1,\notag
\end{align}
yielding (\ref{mainth1-id1-20180319}). Similarly, we can prove (\ref{mainth1-id2-20180319}).
Therefore, the proof of Theorem~\ref{mainthm-1} is completed.\\ \\
%
%
%
{\bf Acknowledgments.} 
This work is partially supported by the research Grant MOST-108-2115-M-007-006-MY2 of Taiwan. 
The author is deeply indebted to Professors Tai-Chia Lin and Chun Liu for exerting an imperceptible influence on his research. He would also like to thank Dr. Chun-Ming Yang for providing numerical simulations of the solutions to equation~(\ref{eq2})--(\ref{bd2}) and equation~(\ref{v-eqn})--(\ref{v-bdy}) with $\eps=10^{-3}$. Finally, the author is grateful to two anonymous referees and the editor for their helpful remarks which improve the exposition of this paper.  


\section{Appendix: Uniqueness result of (\ref{eq1}) with three type boundary conditions}\label{appsec}
\noindent

 In this section, we show the strictly convexness of the functional~(\ref{energy0}).
\begin{proposition}\label{convexthm}
For any $U_1$, $U_2\in{H}^1(\Omega)$ with $U_1{\neq}U_2$, we have
\begin{align}\label{procoveng}
E_{\eps}[tU_1+(1-t)U_2]<tE_{\eps}[U_1]+(1-t)E_{\eps}[U_2],\quad\forall{t\in(0,1)}.
\end{align}
\end{proposition}
\begin{proof}
Since $\int_{\Omega}|\nabla{U}|^2\dxx$ and $\int_{\partial\Omega}(U-a)^2\dsx$ are convex functionals,
it suffices to show that $\widehat{E}_{\eps}[U]:=\log\fint_{\Omega}\cosh{U}\dxx$ is strictly convex.
We need the following elementary inequality:
\begin{align}\label{abinequ}
\begin{cases}
(A+1)^t(B+1)^{1-t}\geq{A}^tB^{1-t}+1,\quad\mathrm{for}\,\,{A},\,B>0\,\,\mathrm{and}\,\,t\in(0,1),\\
\mathrm{and\,\,the\,\,equality\,\,holds\,\,if\,\,and\,\,only\,\,if}\,\,A=B.
\end{cases}
\end{align}

Note that $\frac{1}{t},\frac{1}{1-t}>1$.
Applying (\ref{abinequ})
with $A=e^{2U_1}$ and $B=e^{2U_2}$ and the H\"{o}lder inequality to $\widehat{E}_{\eps}$, one may check that
\begin{equation}\label{ehat-1-app}
  \begin{aligned}
\widehat{E}_{\eps}[tU_1+(1-t)U_2]=&\,\log\fint_{\Omega}\frac{1}{2}e^{-(tU_1+(1-t)U_2)}\left(e^{2(tU_1+(1-t)U_2)}+1\right)\dxx\\
\leq&\,\log\fint_{\Omega}\frac{1}{2}e^{-(tU_1+(1-t)U_2)}\left(e^{2U_1}+1\right)^t\left(e^{2U_2}+1\right)^{1-t}\dxx\\
=&\,\log\fint_{\Omega}\left(\frac{e^{U_1}+e^{-U_1}}{2}\right)^{t}\left(\frac{e^{U_2}+e^{-U_2}}{2}\right)^{1-t}\dxx\\
\leq&\,\log\left(\fint_{\Omega}\frac{e^{U_1}+e^{-U_1}}{2}\dxx\right)^t\left(\fint_{\Omega}\frac{e^{U_2}+e^{-U_2}}{2}\dxx\right)^{1-t}\\
=&\,t\widehat{E}_{\eps}[U_1]+(1-t)\widehat{E}_{\eps}[U_2].
\end{aligned}  
\end{equation}

Moreover, the equality of (\ref{ehat-1-app}) holds if and only if $e^{2U_1}=e^{2U_2}$ 
(by (\ref{abinequ}) and the second line of (\ref{ehat-1-app}))
and $\frac{\cosh{U_1}}{\cosh{U_2}}$ takes a constant value  in $\Omega$ (from the fourth line of (\ref{ehat-1-app}) and the condition for equality to hold). As a consequence, \eqref{ehat-1-app} implies $U_1=U_2$, and
 we have $\widehat{E}_{\eps}[tU_1+(1-t)U_2]<t\widehat{E}_{\eps}[U_1]+(1-t)\widehat{E}_{\eps}[U_2]$ for $U_1\neq{U}_2$. Therefore, we get (\ref{procoveng}) and complete the proof of Proposition~\ref{convexthm}.
\end{proof}

On the other hand, one finds $\ds\inf_{H^1(\Omega)}E_{\eps}\geq|\Omega|\log1=0$
since $\cosh{U}\geq1$.
Along with Proposition~\ref{convexthm}, we may use the Direct method to show that $E_{\eps}$ has a unique minimizer $U^*$ in $H^1(\Omega)$, and $U^*$ is a weak solution of (\ref{eq1})--(\ref{bd1}). Then the standard elliptic regularity theory
immediately shows that the minimizer $U^*\in\mathrm{C}^1(\overline{\Omega})\cap\mathrm{C}^{\infty}(\Omega)$ is a class solution of (\ref{eq1})--(\ref{bd1}) for bounded domain $\Omega$ with smooth boundary. 


We now shall prove the uniqueness of the model~(\ref{eq1}) with three type boundary conditions:
the Robin boundary condition~(\ref{bd1}), the Dirichlet boundary condition
\begin{align}\label{dbd}
U=\widetilde{a}(x)\quad\mathrm{on}\,\,\partial\Omega,
\end{align}
and the Neumann boundary condition
\begin{align}\label{nbd}
\partial_{\vec{n}}U=\widetilde{a}(x)\quad\mathrm{on}\,\,\partial\Omega,
\end{align}
where $\widetilde{a}$ is a smooth function on $\partial\Omega$.
\begin{proposition}\label{corapp}
The model (\ref{eq1}) with the following boundary conditions has a unique classical solution.
\begin{itemize}
\item[(i)]\,\,The Robin boundary condition~(\ref{bd1}).
\item[(ii)]\,\, The Dirichlet boundary condition~(\ref{dbd}).
\item[(iii)]\,\,The Neumann boundary condition~(\ref{nbd}).
\end{itemize}
\end{proposition}
\begin{proof}
The main argument is based on the proof of Theorem~1.1 in \cite{l2014}.
For convenience, we let 
\begin{align*}
\pmb{\mathtt{C}}_{\eps}^i=\left(\fint_{\Omega}\cosh{U^i}\,\dxx\right)^{-1},\quad{i}=1,2.
\end{align*}

\textbf{Proof of (i).} 
Suppose by contradiction that $U^1$ and $U^2$ are two distinct classical solutions of (\ref{eq1})--(\ref{bd1}).
Subtracting (\ref{eq1}) for $U=U^2$ from
that for $U=U^1$, multiplying the result by $U^1-U^2$ and then integrating the expression over $\Omega$,
 one may check that
\begin{align}\label{11-10-928}
-\eps^2&\int_{\Omega}|\nabla(U^1-U^2)|^2\dxx-\frac{\eps}{\gamma}\int_{\partial\Omega}(U^1-U^2)^2\dsx\notag\\
=&\int_{\Omega}\left(\pmb{\mathtt{C}}_{\eps}^1{\sinh{U^1}}-\pmb{\mathtt{C}}_{\eps}^2{\sinh{U^2}}\right)(U^1-U^2)\dxx\notag\\
=&\frac{1}{2}\int_{\Omega}\left(e^{U^1+\log\pmb{\mathtt{C}}_{\eps}^1}-e^{U^2+\log\pmb{\mathtt{C}}_{\eps}^2}\right)\left[(U^1+\log\pmb{\mathtt{C}}_{\eps}^1)-(U^2+\log\pmb{\mathtt{C}}_{\eps}^2)+\log\frac{\pmb{\mathtt{C}}_{\eps}^2}{\pmb{\mathtt{C}}_{\eps}^1}\right]\dxx\notag\\
&+\frac{1}{2}\int_{\Omega}\left(e^{-U^1+\log\pmb{\mathtt{C}}_{\eps}^1}-e^{-U^2+\log\pmb{\mathtt{C}}_{\eps}^2}\right)\left[(-U^1+\log\pmb{\mathtt{C}}_{\eps}^1)-(-U^2+\log\pmb{\mathtt{C}}_{\eps}^2)+\log\frac{\pmb{\mathtt{C}}_{\eps}^2}{\pmb{\mathtt{C}}_{\eps}^1}\right]\dxx\\
\geq&\frac{1}{2}\log\frac{\pmb{\mathtt{C}}_{\eps}^2}{\pmb{\mathtt{C}}_{\eps}^1}\int_{\Omega}\left[\left(e^{U^1+\log\pmb{\mathtt{C}}_{\eps}^1}-e^{U^2+\log\pmb{\mathtt{C}}_{\eps}^2}\right)+\left(e^{-U^1+\log\pmb{\mathtt{C}}_{\eps}^1}-e^{-U^2+\log\pmb{\mathtt{C}}_{\eps}^2}\right)\right]\dxx\notag\\
=&\log\frac{\pmb{\mathtt{C}}_{\eps}^2}{\pmb{\mathtt{C}}_{\eps}^1}\left(\pmb{\mathtt{C}}_{\eps}^1\int_{\Omega}\cosh{U^1}\dxx-\pmb{\mathtt{C}}_{\eps}^2\int_{\Omega}\cosh{U^2}\dxx\right)=0.\notag
\end{align}
Here we have applied the integration by parts and the boundary constraint~$(U^1-U^2)+\gamma\eps\partial_{\vec{n}}(U^1-U^2)=0$ to the left-hand side of (\ref{11-10-928}),
and the elementary inequality $(e^A-e^B)(A-B)\geq0$ (for $A,B\in\mathbb{R}$) to the fourth line of (\ref{11-10-928}).
Note that $\gamma>0$. Thus, (\ref{11-10-928}) gives $\int_{\Omega}|\nabla(U^1-U^2)|^2\dxx=\int_{\partial\Omega}(U^1-U^2)^2\dsx=0$, which immediately implies
\begin{align}\notag
\nabla(U^1-U^2)=0\,\,\mathrm{in}\,\,\Omega,\,\,\mathrm{and}\,\,U^1-U^2=0\,\,\mathrm{on}\,\,\partial\Omega.
\end{align}
Therefore, we get $U^1=U^2$ in $\overline{\Omega}$ (also leads to a contradiction) and complete the proof of~(i).

\textbf{Proof of (ii).}
Suppose that $U^1$ and $U^2$ are two distinct classical solutions of (\ref{eq1}) with the Dirichlet boundary condition~(\ref{dbd}). Following the same argument of (\ref{11-10-928}), we arrive at  
\begin{align}\label{11-10-928-1}
&-\eps^2\int_{\Omega}|\nabla(U^1-U^2)|^2\dxx\notag\\
\geq&\frac{1}{2}\int_{\Omega}\left(e^{U^1+\log\pmb{\mathtt{C}}_{\eps}^1}-e^{U^2+\log\pmb{\mathtt{C}}_{\eps}^2}\right)\left[(U^1+\log\pmb{\mathtt{C}}_{\eps}^1)-(U^2+\log\pmb{\mathtt{C}}_{\eps}^2)\right]\dxx\\
&+\frac{1}{2}\int_{\Omega}\left(e^{-U^1+\log\pmb{\mathtt{C}}_{\eps}^1}-e^{-U^2+\log\pmb{\mathtt{C}}_{\eps}^2}\right)\left[(-U^1+\log\pmb{\mathtt{C}}_{\eps}^1)-(-U^2+\log\pmb{\mathtt{C}}_{\eps}^2)\right]\dxx\geq0.\notag
\end{align}
Hence, in $\Omega$ we must have $\nabla{U}^1=\nabla{U}^2$ and
$$\left(e^{U^1+\log\pmb{\mathtt{C}}_{\eps}^1}-e^{U^2+\log\pmb{\mathtt{C}}_{\eps}^2}\right)\left[(U^1+\log\pmb{\mathtt{C}}_{\eps}^1)-(U^2+\log\pmb{\mathtt{C}}_{\eps}^2)\right]=0,$$
 $$\left(e^{-U^1+\log\pmb{\mathtt{C}}_{\eps}^1}-e^{-U^2+\log\pmb{\mathtt{C}}_{\eps}^2}\right)\left[(-U^1+\log\pmb{\mathtt{C}}_{\eps}^1)-(-U^2+\log\pmb{\mathtt{C}}_{\eps}^2)\right]=0.$$
This implies $U^1+\log\pmb{\mathtt{C}}_{\eps}^1=U^2+\log\pmb{\mathtt{C}}_{\eps}^2$ and $-U^1+\log\pmb{\mathtt{C}}_{\eps}^1=-U^2+\log\pmb{\mathtt{C}}_{\eps}^2$, i.e., $U^1=U^2$ in $\Omega$.
Along with $U^1=U^2$ on $\partial\Omega$, we get $U^1=U^2$ in $\overline{\Omega}$ and complete the proof of~(ii).

It remains to prove (iii).
When both $U^1$ and $U^2$ are solutions of (\ref{eq1}) with the boundary condition~(\ref{nbd}),
it is easy to check that
(\ref{11-10-928-1}) still holds. Hence, 
 we immediately get $U^1=U^2$ in $\overline{\Omega}$ and complete the proof of (iii). 
Therefore, the proof of Proposition~\ref{corapp} is done.
\end{proof}


\begin{thebibliography}{99}

\small


\bibitem{ar1981} {\sc U. Ascher and R.D. Russell}, 
{\em Reformulation of boundary value problems in ``standard" form},
SIAM Rev., \textbf{23} (1981), 238--254.






\bibitem{bm2018} {\sc G.V. Bossa, B.K. Berntson and S. May},
{\em Curvature Elasticity of the Electric Double Layer},
 Phys. Rev. Lett. \textbf{120} (2018), 215502.

\bibitem{cd2001} {\sc D. Carpentier and P. Le Doussal}, {\em Glass transition of a particle in a random potential, front selection in nonlinear renormalization group and entropic phenomena in Liouville and sinh--Gordon models}, 
Phys. Rev. E \textbf{63}  (2001), 026110.

\bibitem{eg2019} \textcolor{red}{{\sc D. Elad, N. Gavish}, {\em Finite domain effects in steady state solutions of Poisson--Nernst--Planck equations}, SIAM J. Appl. Math. \textbf{79} (2019) 1030--1050.}


\bibitem{f1973} {\sc P.C. Fife}, {\em Semilinear elliptic boundary value problems with small parameters},
  Arch. Rational Mech. Anal. \textbf{52} (1973),  205--232.
	
	\bibitem{FLZ} {\sc
V. Fateev, D. Fradkin, S.L. Lukyanov, A. Zamolodchikov and A. Zamolodchikov}, {\em Expectation values of
descendent fields in the sine--Gordon model}, Nuclear Phys. B \textbf{540} (1999), 587--609.
	

	\bibitem{fq2011} \textcolor{red}{{\sc G. Feng, R. Qiao, J. Huang, S. Dai, B.G. Sumpter, V. Meunier},
{\em The importance of ion size and electrode curvature on electrical double layers in ionic liquids},
Phys. Chem. Chem. Phys. \textbf{13} (2011) 1152--1161.}


\bibitem{GMN} {\sc D. Gaiotto , G.W. Moore and A. Neitzke}, {\em Wall-crossing, Hitchin systems, and the WKB approximation},
Adv. Math. \textbf{234} (2013), 239--403.


\bibitem{GT2001} {\sc D. Gilbarg and N.S. Trudinger}, 
{\em Elliptic Partial Differential Equations of Second Order}, Classics in Mathematics, Springer,
Berlin, 2001.

\bibitem{hr2015} {\sc J.L. Hineman, R.J. Ryham},
{\em  Very weak solutions for Poisson--Nernst--Planck system},
 Nonlinear Anal. \textbf{115} (2015) 12--24. 

\bibitem{h1979} {\sc F.A. Howes}, {\em Singularly perturbed semilinear elliptic boundary value problems}, 
Commun. Partial Differ. Equ. \textbf{4} (1979), 1--39.


\bibitem{hl2015-1} {\sc C.-Y. Hsieh, Y. Hyon, H. Lee, T.-C. Lin, C. Liu},
{\em Transport of charged particles: entropy production and maximum dissipation
principle}, J. Math. Anal. Appl. \textbf{422} (2015) 309--336.

\bibitem{hl2015-2} {\sc C.-Y. Hsieh, T.-C. Lin}, 
{\em Exponential decay estimates for the stability of boundary layer solutions to Poisson--Nernst--Planck systems:
one spatial dimension case}, SIAM J. Math. Anal. \textbf{4}7 (2015) 3442--3465.

\bibitem{h2019-1} {\sc C.-Y. Hsieh}, 
{\em Global existence of solutions for the Poisson--Nernst--Planck system with steric effects}, 
Nonlinear Anal. Real World Appl. \textbf{50} (2019) 34--54.

\bibitem{h2019-2} {\sc C.-Y. Hsieh}, 
{\em Stability of radial solutions of the Poisson-Nerns-Planck system in annular domains},
 Discrete Contin. Dyn. Syst. Ser. B \textbf{24} (2019) 2657--2681.

\bibitem{hy2020} {\sc C.-Y. Hsieh and Y. Yu}, 
{\em Debye Layer in Poisson--Boltzmann Model with Isolated Singularities},
Arch Rational Mech Anal \textbf{236} (2020) 289–327. 



\bibitem{jk1990} {\sc M. Jaworski and D. Kaup}, 
{\em Direct and inverse scattering problem associated with the
elliptic sinh--Gordon equation}, Inverse Problems \textbf{6} (1990) 543--556.



\bibitem{ks2015} {\sc A. Khare1 and A. Saxena},
{\em Periodic and hyperbolic soliton solutions of a number of nonlocal nonlinear equations},
 J. Math. Phys. \textbf{56} (2015), 032104.

\bibitem{kas2015} \textcolor{red}{{\sc G. Kikugawa, S. Ando, J. Suzuki, Y. Naruke, T. Nakano, T. Ohara}, {\em Effect of the computational domain size
and shape on the self-diffusion coefficient
in a Lennard-Jones liquid}, J. Chem.
Phys. \textbf{142} (2015) 024503.}

\bibitem{k2007-1} {\sc M.S. Kilic, M.Z. Bazant and A. Ajdari}
{\em Steric effects in the dynamics of electrolytes at large applied voltages. I. Double-layer charging},
Phys. Rev. E \textbf{75} (2007), 021502.




\bibitem{l2014} {\sc C.-C. Lee}, 
{\em The charge conserving Poisson--Boltzmann equations: Existence, uniqueness and maximum principle}, J. Math. Phys. \textbf{55} (2014), 051503.

\bibitem{l2016} {\sc C.-C. Lee},
{\em Asymptotic analysis of charge conserving Poisson--Boltzmann equations with variable dielectric coefficients},
Discrete Contin Dyn. Syst. \textbf{36} (2016) 3251--3276.

\bibitem{l2019} {\sc C.-C. Lee},
{\em Thin layer analysis of a non-local model for the double layer structure},
J. Differential Equations \textbf{266} (2019) 742--802.

\bibitem{lhl2011} {\sc C.-C. Lee, H. Lee, Y. Hyon, T.-C. Lin and C. Liu}, {\em New Poisson--Boltzmann type equations: One-dimensional
solutions}, Nonlinearity \textbf{24} (2011), 431--458.

\bibitem{lhl2016} {\sc C.-C. Lee, H. Lee, Y. Hyon, T.-C. Lin and C. Liu}, {\em Boundary layer solutions of 
Charge Conserving Poisson--Boltzmann equations: One-dimensional case}, Commun. Math. Sci. \textbf{14} (2016), 911--940.

\bibitem{lr2018} {\sc C.-C. Lee and R.J. Ryham},
{\em Boundary asymptotics for a non-neutral electrochemistry model with small Debye length}, 
Z. Angew. Math. Phys. \textbf{69} (2018) 41.




\bibitem{m2002} {\sc R. Messina},
{\em Image charges in spherical geometry: Application to colloidal systems},
J. Chem. Phys.  \textbf{117} (2002), 11062.


\bibitem{mhk2001} {\sc R. Messina, C. Holm and K. Kremer},
{\em Strong electrostatic interactions in spherical colloidal systems},
Phys. Rev. E \textbf{64} (2001), 021405.

\bibitem{MTW} {\sc B.M. McCoy, C.A. Tracy and T.T. Wu}, {\em Painleve functions of the third kind}, 
J. Math. Phys. \textbf{18} (1977), 1058.


\bibitem{n2004} {\sc C. Neri}, 
{\em Statistical mechanics of the N-point vortex system with random intensities on a bounded domain},
Ann Inst. H. Poincar\'{e} Anal. Non Lin\'{e}aire \textbf{21}  (2004), 381--399.



\bibitem{o1968} {\sc R.E. O'Malley}, 
{\em Topics in singular perturbations}, Adv. Math. \textbf{2} (1968), 365--470.





\bibitem{rl2007} {\sc R. Ryham, C. Liu and L. Zikatanov}, {\em Mathematical Models for the Deformation of Electrolyte
Droplets}, Discrete Contin. Dyn. Syst. Ser. B \textbf{8} (2007), 649--661.

\bibitem{rlw2006} {\sc R. Ryham, C. Liu and Z.Q. Wang}, {\em On electro-kinetic fluids: one dimensional configurations},
Discrete Contin. Dyn. Syst. Ser. B \textbf{6} (2006), 357--371.




\bibitem{s2004} {\sc T. Shibata}, {\em The steepest point of the boundary layers of singularly perturbed semilinear elliptic problems}, Tran. Amer. Math. Soc. \textbf{356} (2004), 2123--2135.


\bibitem{s2012} {\sc H. Sugioka}, 
{\em Ion-conserving Poisson--Boltzmann theory}, Phys. Rev. E \textbf{86} (2012), 016318.

\bibitem{t2001} {\sc S. Takeuchi}, 
{\em Positive solutions of a degenerate elliptic equation with logistic reaction},
 Proc. Amer. Math. Soc. \textbf{129} (2001), 433--441.







\bibitem{w2014} {\sc L. Wan, S. Xu, M. Liao, C. Liu and P. Sheng}, {\em Self-consistent approach to global charge neutrality in
electrokinetics: A surface potential trap model}, Phys. Rev. X \textbf{4} (2014), 011042.

\bibitem{wm2011} \textcolor{red}{{\sc N. Wilke, B. Maggio}, {\em Electrostatic field effects on membrane domain segregation
and on lateral diffusion}, Biophys Rev. \textbf{3} (2011) 185--192.}















\end{thebibliography}
\end{document}